\documentclass{article}
\usepackage[english]{babel}
\usepackage{geometry,amssymb}
\usepackage{amsmath}
\usepackage{amssymb}
\usepackage{amsthm}
\usepackage{bbm}
\usepackage{verbatim}
\usepackage{datetime}

\newcommand{\cdummy}{\cdot}
\newcommand{\tmop}[1]{\ensuremath{\operatorname{#1}}}
\newtheorem{definition}{\bf Definition}[section]
\newtheorem{theorem}[definition]{\bf Theorem}
\newtheorem{lemma}[definition]{\bf Lemma}

\newtheorem{proposition}[definition]{\bf Proposition}
\newtheorem{corollary}[definition]{\bf Corollary}
\newtheorem{remark}[definition]{\bf Remark}

\newtheorem{question}[definition]{\bf Question}

\begin{document}
\date{\today}

\title{Topological pressure for conservative $C^1$-diffeomorphisms with no dominated splitting}

\author{Xueming hui}

\maketitle

\begin{abstract} 
    We prove three formulas for computing topological pressure of $C^1$-generic conservative diffeomorphism with no dominated splitting and show the continuity of topological pressure with respect to these diffeomorphisms. We prove for these generic diffeomorphisms that there is no equilibrium states with positive measure theoretic entropy. In particular, for hyperbolic potentials, there is no equilibrium states.
    
    For $C^1$ generic conservative diffeomorphism on compact surfaces with no dominated splitting and $\phi_m(x):=-\frac{1}{m}\log \Vert D_x f^m\Vert, m \in \mathbb{N}$, we show that there exist equilibrium states with zero entropy and there exists a transition point $t_0$ for the one parameter family $\lbrace t \phi_m\rbrace_{t\geq 0}$, such that there is no equilibrium states for $ t \in [0, t_0)$  and there is an equilibrium state for $t \in [t_0,+\infty)$.
    
\end{abstract} 

	
\section{Introduction}
\

Topological pressure is a natural generalization of topological entropy. The variational principle for topological pressure says that $P(f,\phi) = \sup_{\mu \in \mathbb{P}(f)} \lbrace h_{\mu}(f)+\int \phi \ d \mu \rbrace$. Here $f$ is a continuous map from a compact metric space $X$ to itself, $\phi$ is a continuous function from $X$ to $\mathbb{R}$ and is referred to as a potential function or observable, and $\mathbb{P}(f)$ is the set of invariant Borel probability measures of $f$, $P(f,\phi)$ denotes the topological pressure of $(f,\phi)$. If there exists a $\nu \in \mathbb{P}(f)$ such that $P(f,\phi)=h_{\nu}(f)+\int \phi \ d \nu$, then $\nu$ is called an equilibrium state for $(f,\phi)$. In the special case $\phi\equiv 0$, $P(f,0)=h_{\tmop{top}}(f)$ and an equilibrium state is called a measure of maximal entropy for $f$. 

\

Here we briefly introduce some known results for equilibrium states for smooth dynamical system. First, if $f$ is $C^{\infty}$, there exists an equilibrium state. The idea is that in this situation the entropy map $\mu \mapsto h_{\mu}(f)$ is upper semi-continuous \cite{10.2307/1971492} and therefore by weak*-compactness of the space of $f$-invariant Borel probability measures, there is an equilibrium state. More generally, expansivity will guarantee that the entropy map is upper semi-continuous and thus implies the existence of equilibrium states. In particular, hyperbolic diffeomorphisms are expansive. For an Axiom A diffeomorphism, there are many important results, see \cite{bowen2008equilibrium}. For example, for a H\"{o}lder continuous potential function, there are at most finitely many ergodic equilibrium states, and the problem of the existence of a unique equilibrium state is solved by finding a Markov partition. Then using results of David Ruelle on the Gibbs measure on the subshift of finite type \cite{10.2307/2373810}, the result for Axiom A diffeomorphisms follows. 

\

Partially hyperbolic systems are not expansive in general, but when the center direction is one-dimensional, we still have the existence of equilibrium states. This is because they are \emph{entropy expansive} \cite{10.2307/1995978,cowieson_young_2005}, and entropy expansiveness is enough to guarantee the existence of equilibrium states. When the center direction is more than one dimensional, counter examples are constructed by Buzzi and Fisher \cite{1930-5311_2013_4_527}. There are several examples of equilibrium states of partially hyperbolic systems, one of the first is known as \textquotedblleft partially hyperbolic horseshoes\textquotedblright, the existence of equilibrium states was proved by Leplaideur, Oliveira and Rios \cite{leplaideur_oliveira_rios_2011}; the uniqueness results for potentials constant on the center-stable direction were given by Arbieto and Prudente \cite{1078-0947_2012_1_27}; several examples of phase transitions were given by D\'{\i}az, Gelfert, and Rams \cite{Porcupine,D_az_2014,D_az_2011}. For certain partially hyperbolic horseshoes, the uniqueness of equilibrium states for H\"{o}lder potentials with small variation was proved by Rios and Siqueira \cite{rios_siqueira_2018}, and the statistical properties of these equilibrium states were studied by Ramos and Siqueira \cite{Ramos2017}. A related class of partially hyperbolic skew-products with non-uniformly expanding base and uniformly contracting fiber was studied by Ramos and Viana \cite{Ramos_Viana_2017}. Some results on the uniqueness of equilibrium states for partially hyperbolic DA systems on $\mathbb{T}^3$ were proved Crisostomo and Tahzibi \cite{Crisostomo_2019}. 

\

In this paper, we focus on studying the topological pressure and equilibrium states for $C^1$-generic conservative diffeomorphism with no dominated splitting. Using a symbolic coding to determine equilibrium states does not work in this case, since by a result of Buzzi, Crovisier and Fisher \cite{2016arXiv160601765B}, there is no symbolic extension for these diffeomorphisms. Instead, we will prove that for these diffeomorphisms, there is no equilibrium states with positive measure theoretic entropy for any continuous potential function. This is a generalization of results in \cite{2016arXiv160601765B}, where it is shown that there is no measure with maximal entropy for these diffeomorphism. 

\

Before stating our main results, we define some notations and concepts. Let $M$ be a $d_0$ dimensional compact orientable Riemannian manifold and $\mathrm{Diff}^1_\omega(M)$ be the set of diffeomorphisms preserving $\omega$ where $\omega$ is either a volume form or a symplectic form. Throughout this paper, whenever $p$ is a periodic point we denote the period of $p$ by $T(p)$. Let $\lambda_1(f,p)\leq \lambda_2(f,p)\leq ... \leq \lambda_{d_0}(f,p)$ be the logarithms of the absolute values of the eigenvalues of $D_pf^{T(p)}$,  $\lambda_i^+(f,p):=\max (\lambda_i(f,p), 0)$, $\lambda_i^-(f,p):= \max (-\lambda_i(f,p) , 0)$ and $\Delta (f, p):=\min (\sum_{i=1}^{d_0} \lambda_i^+(f,p), \sum_{i=1}^{d_0} \lambda_i^-(f,p))$.  In particular, for conservative system, $\Delta(f,p)=\sum_{i=1}^{d_0} \lambda_i^+(f,p)= \sum_{i=1}^{d_0} \lambda_i^-(f,p)$.

We say that an invariant compact set $\Lambda$
has a \emph{dominated splitting} if there exists a non-trivial decomposition
$TM|_{\Lambda}=E\oplus F$ of the tangent bundle of $M$ over $\Lambda$ in two invariant continuous subbundles, $C>0$ and $0<\lambda <1 $,  such that for all $x\in \Lambda$, all $n\geq 1$ and all unit vectors $u\in E(x)$ and $v\in F(x)$ we have
$$
\|Df^nu\|\leq C \lambda^n \|Df^n v\| .
$$

\medskip

Firstly, we have an abstract result, Theorem \ref{main}, about the lower bound for topological pressure for a given system $(f,\phi)$, where $f$ is a diffeomorphism on a compact Riemannian manifold satisfying certain properties and $\phi:M\mapsto \mathbb{R}$ is continuous. This will give us three formulas for topological pressure for a generic diffeomorphism $f\in \mathcal{E}_{\omega} (M)$, where $\mathcal{E}_{\omega} (M)$ denotes the interior of the set of all diffeomorphisms in $\mathrm{Diff}^1_\omega(M)$ that do not have a dominated splitting on the entire manifold. This is always nonempty\footnote{One can find a diffeomorphism such that for each $1\leq i< d_0$ there exists a periodic point $p$ with period $m$ such that $D_pf^m$ has $d_0$ simple eigenvalues $\lambda_1,\dots,\lambda_{d_0}$ such that $|\lambda_k|<|\lambda_{k+1}|$ for each $k\neq i$ and such that $\lambda_i,\lambda_{i+1}$ are non-real conjugated complex numbers. Then if $f$ has a dominated splitting on the entire manifold, by continuity of the splitting, we have the dimensions of the finest splitting will be constant which contradicts our construction.}.

\medskip

As in \cite{2016arXiv160601765B}, we can relate topological pressure to Lyapunov exponents of periodic orbit and the Birkhoff average of potential function on the periodic orbit. And then by Theorem \ref{main}, we can relate that to the topological pressure of horseshoe. More specifically, we have the following.
\begin{theorem}\label{PressureFormula}
	There exists a residual subset $\mathcal{G}$ of $\mathcal{E}_{\omega} (M)$ such that for any $f\in \mathcal{G}$ and any continuous function $\phi:M\mapsto \mathbb{R}$ we have
	\[ P (f, \phi) = \sup_{p \in \tmop{Per} (f)} \{\Delta (f, p) + \frac{1}{T(p)}  \sum_{i = 0}^{T (p) - 1} \phi (f^i (p))\} = \sup_{\tmop{Horseshoe}\ K} P (f|_K, \phi |_K). \]
\end{theorem}
By the structural stability of horseshoes, we have the lower semi-continuity of $P(f,\phi)$ with respect to $f$. The upper semi-continuity can also be proved on a residual subset, therefore

\begin{theorem}\label{continuity}
There is a residual subset $\mathcal{G}$ in $\mathcal{E}_{\omega} (M)$ such that for any fixed $\phi\in C^0(M,\mathbb{R})$, each $f\in \mathcal{G}$ is a continuity point of the map $f: \mapsto P (f, \phi) $.
\end{theorem}

For any $C^1$-diffeomorphism $f$ of a compact, $d_0$ dimensional, Riemannian manifold $M$ and for each integer $1\leq k\leq d_0$, we define the following quantity, 
\[\sigma_k(f,\phi):=\lim_{n \rightarrow \infty} \sup_{E\in \tmop{Grass}_k(TM)}\frac{1}{n}\left( \log \vert \tmop{Jac}(f^n,E)\vert + S_n \phi(\pi_k(E))\right). \]
Here $\pi_k(E)$ is the base-point of $E$ and $S_n(\phi)$ is the Birkhoff sum of $\phi$.

Then we have the following formula for topological pressure of generic diffeomorphism in $\mathcal{E}_{\omega} (M)$.
\begin{theorem}\label{Pressureformula}
	There exists a residual subset $\mathcal{G}$ of $\mathcal{E}_{\omega} (M)$ such that for any $f\in \mathcal{G}$ and any $\phi\in C^0(M,\mathbb{R})$ we have
	\[ P (f, \phi)=\max_{1\leq k\leq d_0}\sigma_k(f,\phi). \]
\end{theorem}

In \cite{2016arXiv160601765B}, Buzzi, Crovisier and Fisher prove that for a generic diffeomorphism in $\mathcal{E}_{\omega} (M)$, there is no measure with maximal entropy. Similarly, we prove that for these diffeomorphism and any continuous potential function, there is no equilibrium states with positive measure theoretic entropy. 
\begin{theorem}\label{No-ES-positive-entropy}
	There exists a residual subset $\mathcal{G}$ of $\mathcal{E}_{\omega} (M)$ such that for every $f\in \mathcal{G}$ and $\phi \in C^0(M,\mathbb{R})$ we have, 
	any equilibrium state of $(f,\phi)$ has zero measure theoretic entropy.
\end{theorem}

We know that a K-automorphism must have positive measure theoretic entropy \cite{MR0217258}, therefore by Theorem \ref{No-ES-positive-entropy}, we have:
\begin{corollary}
	There exists a residual subset $\mathcal{G}$ of $\mathcal{E}_{\omega} (M)$ such that for every $f\in \mathcal{G}$ and $\phi \in C^0(M,\mathbb{R})$ we have, if $\mu$ is an equilibrium state, then $(f,\mu)$ is not a K-automorphism.
\end{corollary}

Note that it is always true that 
\[ \sup_{\mu \in \mathbb{P}_{\tmop{erg}}(f)} \int \phi \ d\mu \leq P(f,\phi). \]

Following \cite{Inoquio-Renteria2012}, the potential $\phi$ will be called a \emph{hyperbolic potential} for $f$ if we have strict inequality in the above inequality.
By $ \sup_{\mu \in \mathbb{P}_{\tmop{erg}}(f)} \int \phi \ d\mu \leq \sup \phi$ and $\inf \phi +h_{\tmop{top}}(f) \leq P(f,\phi)$, the following corollary is immediate. 

\begin{corollary}
	There exists a residual subset $\mathcal{G}$ of $\mathcal{E}_{\omega} (M)$ such that for every $f\in \mathcal{G}$ and hyperbolic potential $\phi$ , there is no equilibrium states.
In particular, if  $\sup \phi -\inf \phi < h_{\tmop{top}}(f)$, then $\phi$ is a hyperbolic potential function. 
\end{corollary}

\

Based on the above general results for a compact Riemannian manifold of arbitrary dimension, we have some results for conservative non-Anosov diffeomorphisms of compact surfaces. \footnote{For an area-preserving diffeomorphism, if it has a dominated splitting on a compact invariant set, then the splitting is hyperbolic. So $\mathcal{E}_{\omega} (M)$ is the interior of the set of all non-Anosov diffeomorphisms on $M$.}
We first recall some results for conservative diffeomorphisms on compact surfaces.

Let $p$ be a periodic point of period $m$ for a conservative diffeomorphism $f$ on the compact surface $M$. We will call $p$ an \emph{elliptic periodic point} if the linear operator $D_pf^m: T_pM \mapsto T_pM$ has only eigenvalues of norm $1$ but not equal to $1$. In the conservative setting, this means that it has two conjugate complex eigenvalues on the unit circle. By a result of S. Newhouse in \cite{10.2307/2374000}, we know that for a generic diffeomorphism in $\mathcal{E}_{\omega} (M)$, the set of all elliptic periodic orbits are dense in $M$ and each elliptic orbit can be approximated by hyperbolic periodic orbits. This shows that there are many points with only zero Lyapunov exponents. A result by J. Bochi \cite{bochi_2002} shows that for $C^1$ generic conservative diffeomorphismon on surfaces, either the diffeomorphism is Anosov or the set of points with only zero Lyapunov exponents has full Lebesgue measure.This result was announced by R. $\tmop{Ma \tilde{n}  \acute{e}}$ around 1980 without a published proof. Here we address the question of whether there is an ergodic equilibrium state with only zero Lyapunov exponents.

\
By Theorem \ref{No-ES-positive-entropy}, for a generic diffeomorphism in $\mathcal{E}_{\omega} (M)$ with a hyperbolic potential function, there is no equilibrium states. 
It should not be surprising that for some specific potential functions, there exist equilibrium states. That is for these diffeomorphism, non-hyperbolic potential functions exist. For example, for $\phi_m(x):=-\frac{1}{m}\log \Vert D_xf^m\Vert, m\in \mathbb{N}$, we have that for $f$ in a residual subset of $\mathcal{E}_{\omega} (M)$, we can find $m$ depending on $f$, such that there exists equilibrium states for $(f,\phi_m)$. We also prove a result on the phase transition of the family of  potential function $t\phi _m$ for $t\geq 0$. More details on the results on compact surfaces can be found in Section \ref{zeroExponent}.

\subsection*{Outline}

This paper is outlined as follow. In Section 2, we give the preliminaries for the paper. Section 3 is devoted to a result on a non-trivial lower bound of the topological pressure of a horseshoe created by a $C^1$-perturbation near a periodic orbit with weak domination and large period. In Section 4, we prove three formulas for topological pressure for $C^1$-generic conservative diffeomorphism with no dominated splitting. In particular, Theorem \ref{PressureFormula} is proved in subsection \ref{PH}, Theorem \ref{continuity} is proved in subsection \ref{C} and Theorem \ref{Pressureformula} is given in subsection \ref{Jac}. In Section 5, we give the proof of non-existence of equilibrium states with positive entropy for these diffeomorphisms (Theorem \ref{No-ES-positive-entropy}). It should be mentioned that many results in section 4 and 5 are generalizations of parallel results in \cite{2016arXiv160601765B}. Finally, the last section, Section 6 contains some results for surface diffeomorphisms. 

\subsection*{Acknowledgments}
The author would like to thank Todd Fisher for introducing this problem to him and for useful discussions, Sylvain Croviser for answering some questions about the paper.

\

\section{Preliminaries}

\

In this section ,we recall some basic notations, concepts and theorems that will be used. Specifically, we review weak forms of hyperbolicity, topological pressure and perturbation tools for $C^1$ dynamics.

Let $M$ be a compact orientable connected boundaryless Riemannian manifold with dimension $d_0$, $f$ be a $C^1$ diffeomorphism and $\Lambda$ be a compact $f$-invariant subset of $M$. $\Lambda$ is said to be a \emph{hyperbolic set} for $f$, if there is an $f$-invariant splitting $TM|_{\Lambda}=E\oplus F$ and constants $C>0$, $0<\lambda<1$ such that the following hold,
\[ \Vert D_xf^n |_E \Vert \leq C \lambda^n ,\]
\[ \Vert D_xf^{-n} |_F \Vert \leq C \lambda^n,\]
for any non-negative integer $n$ and any $x\in \Lambda$.

If $\Lambda$ is a hyperbolic set for a $C^1$ diffeomorphism $f$, and $x\in\Lambda$, the \emph{stable manifold of $x$} is set given by
$W^s(x)=\{ y\in M\, :\, d(f^nx, f^ny)\to 0, n\to +\infty\}$. Similarly, we can define the \emph{unstable manifold} of $x$ as $W^u(x)=\{ y\in M\, :\, d(f^nx, f^ny)\to 0, n\to -\infty\}$. 

A hyperbolic set $\Lambda$ is \emph{locally maximal} if there is an open neighborhood $U$ of $\Lambda$ such that $\Lambda=\bigcap_{n\in\mathbb{Z}}f^n(U)$.  $f$ is called \emph{transitive} on an invariant set $\Lambda$ if there is a point $x$ in $\Lambda$ such that the closure of its orbit is $\Lambda$.
A set $K$ is called a \emph{horseshoe} for a diffeomorphism $f$ if the following hold.
\begin{itemize}
\item $K$ is a locally maximal hyperbolic set;
\item $K$ is homeomorphic to the Cantor set; and 
\item $f$ is transitive on $K$.

\end{itemize}

The concept of a dominated splitting is a generalization of hyperbolicity. 
\begin{definition}
	Let $N\geq 1$ be a natural number,
	we say that an invariant compact set $\Lambda$
	has an \textbf{N-dominated splitting} if there exists a non-trivial invariant decomposition
	$TM|_{\Lambda}=E\oplus F$ such that for every $x\in \Lambda$, all $n\geq N$ and all unit vectors $u\in E(x)$ and $v\in F(x)$ we have,
	\[\|Df^nu\|\leq \|Df^n v\|/2. \]
\end{definition}
It says that one subbundle in the splitting is dominated by another one, but we may not have the uniform contraction and expansion in each subbundle. This definition of dominated splitting is equivalent to the definition given in the introduction. It will be useful when we apply some perturbation techniques. In particular, we will use this definition in Theorem \ref{main}. 

If $p$ is a periodic point for $f$,
we denote by $T(p)$ the period of $p$ and by $\mathcal{O} (p)$  its orbit.
In this case we say that $p$ is a \emph{saddle} if the orbit $\mathcal{O} (p)$ is hyperbolic and both $E$ and $F$ are nontrivial. 

\begin{definition}\label{def-TNweak} 
	Let $T, N$ be two natural numbers.
	We say that a periodic point $p$ is \textbf{$\mathbf{T, N}$-weak} if $T(p)\geq T$ and $\mathcal{O} (p)$ has no $N$-dominated splitting.
\end{definition}

\begin{definition}
	A diffeormorphism $f$ with \textbf{no dominated splitting} on a compact manifold $M$ is one such that the whole manifold $M$ does not have an N-dominated splitting for every $N \geq 1$.
\end{definition}

We say that two hyperbolic periodic orbits $O_1,O_2$ are \emph{homoclinically related}
if $W^s(O_1)$ intersects $W^u(O_2)$ transversely, and $W^u(O_1)$ intersects $W^s(O_2)$ transversely. Here $W^s(O_1)$ is the union of the stable manifolds of every point in $O_1$ and $W^u(O_1)$ the union of the unstable manifolds of every point in $O_1$. $W^s(O_2)$ and $W^u(O_2)$ are defined similarly.
The \emph{homoclinic class} $H(p)$ of a hyperbolic periodic point $p$ is the closure of the set of all
hyperbolic periodic orbits that are homoclinically related to $\mathcal{O}(p)$. 

If $\omega$ is a volume or a symplectic form on $M$, one denotes by Diff$^1_\omega(M)$ the subspace of diffeomorphisms which preserve $\omega$. We say that a diffeomorphism is \emph{conservative} if it is in Diff$^1_\omega(M)$ for a volume or symplecetic form $\omega$. 

\

For each natural number $n$, we define a metric $d_n$ on a compact metric space $(X,d)$ by
\[ d_n(x,y):=\max_{0\leq i\leq n-1} d(f^i(x),f^i(y)). \]
An $(n,\epsilon)$-Bowen ball $B_f(x,n,\epsilon)$ is an open ball with radius $\epsilon$ in the $d_n$ distance centered at some point $x\in X$. A set $E$ is said to be $(n,\epsilon)$-spanning if $X\subset \bigcup_{x\in E} B_f(x,n,\epsilon)$.
\

Topological pressure is an analog of topological entropy. Let $f: X\mapsto X$ be continuous and $\phi$ be the potential function. Denote by $Q_n(f,\phi,\epsilon)$ the infimum of $\sum_{x\in F} \exp (S_n \phi)(x)$ over all the $(n,\epsilon)$ spanning sets $F$ for $X$, where $S_n \phi (x)$ denote the Birkhorff sum of $f$ at $x$. Then the following limit exist and is called the topological pressure of the system $(f,\phi)$. 
\[ P(f,\phi):=\lim_{\epsilon \rightarrow 0^+}\limsup_{n\rightarrow \infty} \frac{1}{n} \log Q_n(f,\phi,\epsilon) \]

We denote by $\mathbb{P}(f)$ the set of all Borel probability measures that $f$ preserves,
and $\mathbb{P}_{\tmop{erg}}(f)$ the set of those which are ergodic.

There is a variational principle for topological pressure just like that of entropy. It says that $P(f,\phi) = \sup_{\mu \in \mathbb{P}(f)} \lbrace h_{\mu}(f)+\int \phi \ d \mu \rbrace=\sup_{\mu \in \mathbb{P}_{\tmop{erg}}} \lbrace h_{\mu}(f)+\int \phi \ d \mu \rbrace$. For a proof of Variational Principle see \cite{walters2000introduction}. 

\

Let $f\in \mathrm{Diff}_{\omega}^1(M)$ and $\mu\in\mathbb{P}(f)$.  Then by Oseledet's theorem, there is a $\mu$ full measure set on which we have a decomposition $T_x M = E_1 (x) \oplus \cdots E_{k(x)}(x)$ such that for any non-zero $v\in E_i(x)$
$$\lim_{n\to \infty} \frac{\log \|D_x f^n v\|}{n}=\lambda_i(x).$$
where $\lambda_1(x)\leq\lambda_2(x)\leq\dots\leq\lambda_{d_0}(x)$. The numbers $\lambda_i(x)$ are called \emph{Lyapunov exponents of $x$}.
If $\mu$ is an ergodic measure then $\lambda_i$ is constant almost everywhere for each $i$. 

\
\paragraph{\textbf{Perturbative tools}}

\ 

In order to prove the generic properties, we will need to know how to perturb a diffeomorphism in the $C^1$-topology. In this paper, the main perturbative tools we will use are Franks' Lemma for both conservative and non-conservative $C^1$-diffeomorphism and  $\tmop{Ma \tilde{n}  \acute{e}}$'s ergodic closing lemma. 

\begin{definition}
	Let $u\in T_xM$ and $v\in T_yM$, define
	\[d(u,v):= \inf_{\gamma} \|u-\Gamma_\gamma v\|+\text{Length}(\gamma), \] 
	where $\gamma$ is a $C^1$-curve  connecting $x$ and $y$ and $\Gamma_\gamma$ is the parallel transport with respect to the Levi-Civita connection along $\gamma$.
\end{definition}

\begin{definition}
	The $C^1$-topology on the space $\mathrm{Diff}_{\omega}^1(M)$ is the topology induced by the following metric
	\[ d_{C^1}(f,g)=\sup_{v\in T^1M} \max\big( d(Df (v),Dg (v)), \ d(Df^{-1}(v),Dg^{-1}(v))\big). \]
	for any $f,g$ in $\mathrm{Diff}_{\omega}^1(M)$
\end{definition}
 Let $f,g \in \mathrm{Diff}_{\omega}^1(M)$, $g$ is said to be an \emph{$\epsilon$-perturbation} of $f$ if $d_{C^1}(g,f)<\epsilon$.
A property is said to hold \emph{robustly} if it persists under $\epsilon$-perturbations for some $\epsilon>0$.

\begin{definition}\label{def-evNpert}
	Let $f\in \mathrm{Diff}_{\omega}^1(M)$, $X\subset M$ be a finite set, $V$ be a neighborhood of $X$, and $\epsilon>0$.
	A diffeomorphism $g\in \mathrm{Diff}_{\omega}^1(M)$ is an \emph{$(\epsilon, V, X)$-perturbation} of $f$ in $\mathrm{Diff}_{\omega}^1(M)$ if $d_{C^1}(f,g)<\epsilon$ and $g(x)=f(x)$ for all $x$ outside of $V\setminus X$.
\end{definition}

We will use the following strengthening of the classical Franks' lemma\cite{11274e1bc2034234a080ae32bc3e6f0a}.

\begin{theorem}[Franks' lemma with linearization \cite{2017Nonli..30.3613B}]\label{t.linearize}
	Let $f\in \mathrm{Diff}_{\omega}^1(M)$, $\epsilon>0$ small, $X\subset M$ be a finite set, $\chi\colon V\to \mathbb{R}^{d_0}$ be a chart with $X\subset V$ and $\chi^{\prime}\colon f(V)\to \mathbb{R}^{d_0}$ be a chart for $f(V)$. 
	For $x\in X$, if $A_x\colon T_xM\to T_{f(x)}M$ be a linear map such that 
	\[\max(\|A_x-Df(x)\|, \|A^{-1}_x-Df^{-1}(x)\|)<\epsilon/2.\]
	Then there exists an $(\epsilon, V,X)$-perturbation $g$ of $f$ in $\mathrm{Diff}_{\omega}^1(M)$ such that for each $x\in X$ the map $\chi^{\prime}\circ g \circ \chi^{-1}$ is linear in a neighborhood of $\chi(x)$ and $Dg(x)=A_x$. 
\end{theorem}

This strengthening of Franks' Lemma allow us to perturb a diffeomorphism locally. This will be convenient for our construction. 

By Riesz Representation Theorem for Borel measures, the space of complex Borel measures on $M$ can be identified to the dual space of the space of continuous function on $M$, this gives a natural topology on the space of complex Borel measures on $M$. 
\begin{definition}
	For a sequence of complex Borel measures $\mu_n$, and a fixed complex Borel measure $\mu$, we say that $\mu_n$ converge to $\mu$ in the vague topology if the following holds 
	\[ \int \phi \ d\mu_n \rightarrow \int \phi\ d\mu, \forall \phi\in C^0(M), \]
	where $C^0(M)$ is the space of continuous function on $M$. 
\end{definition}
One can also show that this topology is metrizable, see \cite{walters2000introduction}. 

Ma\~n\'e's ergodic closing lemma \cite{10.2307/2007021} says that for any $C^1$-diffeomorphism $f$ and any $f$-invariant
ergodic measure of it, there exists a small $C^1$-perturbation $g$ of $f$ having a periodic orbit $O$ that is close to
$\mu$ in the vague topology. 
In~\cite[Proposition 6.1]{Abdenur2011}, the following stronger version of ergodic closing lemma is proved and the proofs also work in conservative case. 

\begin{theorem}[Ergodic closing lemma]\label{t.ergodic-closing}
	For any $C^1$-diffeomorphism $f$ in $\mathrm{Diff}^1_\omega(M)$ and any $\mu$ in $\mathbb{P}_{\tmop{erg}}(f)$,
	there is a $C^1$-perturbation $g$ of $f$ in $\mathrm{Diff}_{\omega}^1(M)$ having a periodic orbit $O$ close to
	$\mu$ in the vague topology and Lyapunov exponents at the points in $O$
	are close to that of $\mu$ for $f$.
	Moreover $f$ and $g$ coincide outside a small neighborhood of the support of $\mu$.
\end{theorem}
The corollary below for $C^1$-generic (conservative or dissipative) diffeomorphism is immediate. 
\begin{corollary}\label{c.ergodic-closing}
	For any generic diffeomorphism $f$ in $\mathrm{Diff}^1_\omega(M)$,
	and any $\mu$ in $\mathbb{P}_{\tmop{erg}}(f)$,
	there is a sequence of periodic orbits $(O_n)$ which converges to $\mu$
	in the vague topology and whose Lyapunov exponents converge to those of $\mu$.
\end{corollary}
 
 This will give the following convenient proposition to spread a periodic orbit in a horseshoe without changing the Lyapunov exponents too much.
 
\begin{proposition}\label{p.control-exponents}~\cite[Theorem 3.10]{Abdenur2011}
	For any generic diffeomorphism $f$ in $\mathrm{Diff}^1_\omega(M)$,
	horseshoe $K$, periodic orbit $O\subset K$ and $\epsilon>0$,
	there is a periodic orbit $O'\subset K$
	that is $\epsilon$-close to $K$ in the Hausdorff distance
	and the set of Lyapunov exponent at the points in $O'$ is $\epsilon$-close to that of $O$.
\end{proposition}

\section{A non-trivial lower bound for Topological pressure }

As in \cite{2016arXiv160601765B}, we create a horseshoe from a hyperbolic periodic point with large period and weak domination after a local $C^1$ perturbation. The topological pressure of the horseshoe after the perturbation will be bounded by the sum of positive Lyapunov exponents plus the arithmetic average of the value of $\phi$ on the orbit of the periodic point with an arbitrarily small error. 

\begin{theorem}\label{main}
	Given any integer $d_0 \geq 2$ and numbers $C > 0$ and $\epsilon > 0$, there exist integers $N, T \geq 1$ with the following property.
	For any closed $d_0$-dimensional Riemannian manifold $M$, any diffeomorphism $f \in \mathrm{Diff}_{\omega}^1(M)$ with $\tmop{Lip}(f), \tmop{Lip} (f^{- 1})
	\leq C$, any $\phi : M\mapsto \mathbb{R}$ continuous, any $T, N$-weak periodic point $p$ with period $T(p)$ and any constant $\delta > 0$, there exists $\rho>0$ and an $(\epsilon, V, \mathcal{O} (p))$-perturbation $g$ of $f$ in $\mathrm{Diff}_{\omega}^1(M)$, where $V:=\bigcup_{i=0}^{T(p)-1}B(f^i(p),\rho)$ and $B(f^i(p),\rho)$ is the open ball centered at $f^i(p)$ with radius $\rho >0$, such that
	\begin{itemize}
		\item for any $x$ in $\mathcal{O} (p)$, $D_xf = D_xg$ and
		
		\item $g$ contains a horseshoe $K \subset V$ such that
		
		\
		
		$P (g|_K, \phi |_K) \geq \Delta (g, p) + \frac{1}{T (p)}  \sum_{i = 0}^{T
			(p) - 1} \phi (g^i (p)) - \delta$.
	\end{itemize}

\end{theorem}

\begin{remark}
	\textup{The perturbation used in this theorem is the same as in the main theorem of \cite{2016arXiv160601765B}. The idea is that we can create a horseshoe in $V$ for arbitrary small $\rho >0$. This will make the variation of $\phi$ (that is $\sup \phi -\inf \phi$) near each point of the periodic orbit as small as we want. Therefore, the potential function $\phi$ will be almost locally constant, and we can use the main theorem of \cite{2016arXiv160601765B} to get the lower bound. }
\end{remark}

\begin{proof}
  Take $\rho$ to be small enough such that $B(f^i(p),\rho) \cap B(f^j(p),\rho)=\emptyset$ for any $i\neq j$ and $\sup_{x,y \in B(f^i(p),\rho)} | \phi (x)-\phi(y)| < \delta$ for each $i$.
  We will use a parallel result for topological entropy in ~\cite[Theorem 4.1]{2016arXiv160601765B}. In \cite{2016arXiv160601765B}, it is shown that there exists an $(\epsilon, V, \mathcal{O} (p))$-perturbation $g$
  satisfying item 1 of our theorem and a horseshoe $K\subset V$ of $g$ such that
  \begin{equation*}
  	 h_{top} (g|_K) \geq \Delta (g, p).
  \end{equation*}
    We will prove that this $g$ works for our
  theorem.
  
  Let $E_{n, \epsilon}$ be a $(n, \epsilon)$-spanning set for $K$, then we
  have,
  \[ \sum_{x \in E_{n, \epsilon}} e^{S_n \phi (x)} \geq \sum_{i = 0}^{T (p)
     - 1} m_i e^{S_n \phi (g^i (p)) - n \delta}, \]
  where $m_i$ is the number of points in $E_{n, \epsilon}\cap B(g^i(p),\rho)$. Then
  \begin{eqnarray*}
    \sum_{i = 0}^{T (p) - 1} m_i e^{S_n \phi (g^i (p)) - n \delta}  &=& 
    \#E_{n, \epsilon}  \sum_{i = 0}^{T (p) - 1} \frac{m_i}{\#E_{n, \epsilon}}
    e^{S_n \phi (g^i (p)) - n \delta}\\
    \inf_{E_{n, \epsilon}}   \sum_{i = 0}^{T (p) - 1} m_i e^{S_n \phi (g^i
    (p)) - n \delta} &\geq& \inf_{E_{n, \epsilon}} \#E_{n, \epsilon} \times
    \inf_{E_{n, \epsilon}}  \sum_{i = 0}^{T (p) - 1} \frac{m_i}{\#E_{n,
    \epsilon}} e^{S_n \phi (g^i (p)) - n \delta}\\
    \frac{1}{n} \log \inf_{E_{n, \epsilon}}  \sum_{i = 0}^{T (p) - 1} m_i e^{
    S_n \phi (g^i (p)) - n \delta} &\geq& \frac{1}{n} \log \inf_{E_{n,
    \epsilon}} \#E_{n, \epsilon} + \frac{1}{n} \log \inf_{E_{n, \epsilon}} 
    \sum_{i = 0}^{T (p) - 1} \frac{m_i}{\#E_{n, \epsilon}} e^{S_n \phi (g^i
    (p)) - n \delta}. 
  \end{eqnarray*}
  Since $\log (\cdummy)$ is increasing, we have
  \[ \log \inf_{E_{n, \epsilon}}  \sum_{i = 0}^{T (p) - 1} \frac{m_i}{\#E_{n,
     \epsilon}} e^{S_n \phi (g^i (p)) - n \delta} = \inf_{E_{n, \epsilon}}
     \log \sum_{i = 0}^{T (p) - 1} \frac{m_i}{\#E_{n, \epsilon}} e^{S_n
     \phi (g^i (p)) - n \delta}. \]
  By the concavity of $\log (\cdummy)$,
  \[ \log \sum_{i = 0}^{T (p) - 1} \frac{m_i}{\#E_{n, \epsilon}} e^{S_n
     \phi (g^i (p)) - n \delta} \geq \sum_{i = 0}^{T (p) - 1}
     \frac{m_i}{\#E_{n, \epsilon}} S_n \phi (g^i (p)) - n \delta. \]
  Notice that, given any $\alpha >0$, there is $N\in \mathbb{N}$, such that for any $n>N$, we have,
  \[ \left| \frac{1}{n} S_n \phi (g^i (p)) - \frac{1}{T (p)} 
  \sum_{j = 0}^{T (p) - 1} \phi (g^j (p)) \right| <\frac{\alpha}{2} \]
  and we can find $F_{n,\epsilon}$, a $(n, \epsilon)$-spanning set for $K$ such that,
  \[ \left| \sum_{i = 0}^{T (p) - 1} \frac{m_i}{\#F_{n,
  		\epsilon}}  \frac{1}{n} S_n \phi (g^i (p)) -\inf_{E_{n, \epsilon}}\sum_{i = 0}^{T (p) - 1} \frac{m_i}{\#E_{n,
  		\epsilon}}  \frac{1}{n} S_n \phi (g^i (p))\right| < \frac{\alpha}{2}. \]
  	Hence,
  \begin{eqnarray*}
     &\ &\left| \inf_{E_{n, \epsilon}}\sum_{i = 0}^{T (p) - 1} \frac{m_i}{\#E_{n,\epsilon}}  \frac{1}{n} S_n \phi (g^i (p)) - \frac{1}{T (p)} \sum_{j = 0}^{T (p) - 1} \phi (g^j (p)) \right| \\
   &<& \frac{\alpha}{2}+\left| \sum_{i = 0}^{T (p) - 1} \frac{m_i}{\#F_{n,\epsilon}}  \frac{1}{n} S_n \phi (g^i (p)) - \frac{1}{T (p)} \sum_{j = 0}^{T (p) - 1} \phi (g^j (p)) \right| \\
   &\leq&  \frac{\alpha}{2}+ \left| \sum_{i = 0}^{T (p) - 1} \frac{m_i}{\#F_{n,\epsilon}} \left(  \frac{1}{n} S_n \phi (g^i (p)) - \frac{1}{T (p)} \sum_{j = 0}^{T (p) - 1} \phi (g^j (p))\right) \right| \\
  &\leq& \frac{\alpha}{2}+ \sum_{i = 0}^{T (p) - 1} \frac{m_i}{\#F_{n,\epsilon}}  \left| \frac{1}{n} S_n \phi (g^i (p)) - \frac{1}{T (p)} \sum_{j = 0}^{T (p) - 1} \phi (g^j (p)) \right| \\
  &<& \alpha.
  \end{eqnarray*}

  That is,
  \[\lim_{n \rightarrow \infty } \inf_{E_{n, \epsilon}}\sum_{i = 0}^{T (p) - 1} \frac{m_i}{\#E_{n,\epsilon}}  \frac{1}{n} S_n \phi (g^i (p))  = \frac{1}{T (p)} \sum_{j = 0}^{T (p) - 1} \phi (g^j (p)).\]
  Finally we have,
  \begin{eqnarray*}
    P (g|_K, \phi |_K) & = & \lim_{\epsilon \rightarrow 0^+} \liminf_{n
    \rightarrow \infty}  \frac{1}{n} \log \inf_{E_{n, \epsilon}}  \sum_{x
    \in E_{n, \epsilon}} e^{S_n \phi (x)} \\
    & \geq & h_{top} (g|_K) + \frac{1}{T (p)}  \sum_{i = 0}^{T (p) - 1} \phi
    (g^i (p)) - \delta \\
    & \geq & \Delta (g, p) + \frac{1}{T (p)}  \sum_{i = 0}^{T (p) - 1} \phi
    (g^i (p)) - \delta.
  \end{eqnarray*}
\end{proof}

\section{Formulas for topological pressure}

\

In this section, we prove three formulas for topological pressure. This enable us to represent topological pressure using some information about periodic orbits, horseshoes and Lyapunov exponents. With these formulas, we will be able to prove the non-existence of equilibrium states with positive entropy for $C^1$-generic conservative diffeomorphism.
\subsection{Periodic orbits and horseshoes}\label{PH}

\ 

By Ruelle's inequality \cite{Ruelle1978}, we have for any ergodic measure $\mu$ that
\[ h_{\mu} (f) \leq \sum_{i = 1}^{d_0} \lambda_i^+ (\mu) . \]
We can apply this to $f^{- 1}$ to get
\[ h_{\mu} (f^{- 1}) \leq \sum_{i = 1}^{d_0} \lambda_i^- (\mu) . \]
Since $h_{\mu} (f) = h_{\mu} (f^{- 1})$, we have
\[ h_{\mu} (f) \leq \min \left( \sum_{i = 1}^{d_0} \lambda_i^+ (\mu), \sum_{i
   = 1}^{d_0} \lambda_i^- (\mu) \right) =: \Delta (f, \mu). \]
We will only consider the case of conservative diffeomorphisms. In this case, 
\begin{equation*}
	\sum_{i = 1}^{d_0} \lambda_i^+ (\mu)= \sum_{i= 1}^{d_0} \lambda_i^- (\mu)= \Delta (f, \mu).
\end{equation*}

By the Variational Principle
\[ \left. P (f, \phi) = \sup_{\mu \in \mathbb{P}_{\tmop{erg}} (f)} \left( h_{\mu}(f) + \int \phi d \mu \right) \leq \sup_{\mu \in \mathbb{P}_{\tmop{erg}}(f)} \left( \Delta (f, \mu) + \int \phi d \mu \right) \right. .\]
Using Corollary \ref{c.ergodic-closing}, we have the following, 
\begin{lemma}\label{one}
	For a $C^1$-generic conservative diffeomorphism $f$,
	\[ P (f, \phi) \leq \sup_{p \in \tmop{Per} (f)} \left( \Delta (f, p) +
	\frac{1}{T (p)}  \sum_{i = 0}^{T (p) - 1} \phi (f^i (p)) \right) . \]
\end{lemma}

We will use the following notation,
\[ \Delta_{\phi}(f,p):= \Delta (f, p) + \frac{1}{T (p)}  \sum_{i = 0}^{T (p) - 1} \phi (f^i (p)) \]
for a periodic orbit $p$.

\begin{lemma}\label{ContinuitPoints}
  For any fixed continuous function $\phi:M\mapsto \mathbb{R}$, the map defined by $f \mapsto \Delta_{\phi} (f):=\sup_{p \in \tmop{Per} (f)}  \Delta_{\phi}(f,p)$ from
  $\tmop{Diff}^1_{\omega} (M)$ to $\mathbb{R}$ has a dense $G_{\delta}$ set of
  continuity points.
\end{lemma}

\begin{proof}
  For each $\alpha \in \mathbb{Q}$, let $U_{\alpha}^+$ (resp. $U_{\alpha}^-$) be the set of diffeomorphisms $f$ such that for any $ g$ sufficiently $C^1$ close to $f$, we have $\Delta_{\phi} (g)>a$ (resp. $\Delta_{\phi} (g)<a$). 
  
  \ 
  
  For any $f$ with $ \Delta_{\phi} (f)\geq \alpha$, let that $p$ is a periodic point of $f$ such that 
  \begin{equation*}
  	 \Delta_{\phi} (f)- \Delta_{\phi}(f,p)< \epsilon/2.
  \end{equation*}
   By Franks' Lemma (Theorem \ref{t.linearize}), we can make a perturbation $h$ of $f$ around the orbit of $p$ to make $ \Delta_{\phi}(h,p)> \alpha $. By taking $\epsilon$ arbitrarily small, $h$ can be chosen to be arbitrarily close to $f$ in the $C^1$ topology. $p$ can be assumed to be a hyperbolic periodic point for $h$. The structural stability of hyperbolic periodic orbits and the $C^1$-smoothness of $h$ implies that $h$ is in $U_{\alpha}^+$. Therefore, $f$ is in $\overline{U_{\alpha}^+}$. 
  
  \ 
  
  For those $f$ such that $\Delta_{\phi}(f)<\alpha$ and $f$ is not in $U_{\alpha}^-$, we can find a sequence of diffeomorphisms $h_n\in \tmop{Diff}^1_{\omega} (M)$ converging to $f$ such that $\Delta_{\phi}(h_n)\geq \alpha$ by the definition of $U_{\alpha}^-$. We know that $h_n \in \overline{U_{\alpha}^+}$. Hence, $f=\lim_{n\rightarrow \infty}h_n \in \overline{U_{\alpha}^+}$. This proves that for any $f\in \tmop{Diff}^1_{\omega} (M)$, $f\in \overline{U_{\alpha}^+ \cup U_{\alpha}^-}$. 
  
  \ 
  
  So $U_{\alpha}:=U_{\alpha}^+ \cup U_{\alpha}^-$ is open and dense in $\mathcal{E}_{\omega} (M)$. If $f\in \bigcap _{\alpha \in \mathbb{Q}}U_{\alpha}$, $  \alpha_1,\  \alpha_2 \in \mathbb{Q}$, such that $\alpha_1<\Delta_{\phi}(f)<\alpha_2$, then we must have $f\in U_{\alpha_1}^+\cap U_{\alpha_2}^-$. Hence, for any $ g$ sufficiently $C^1$ close to $f$, we have $\alpha_1<\Delta_{\phi}(g)<\alpha_2$. This proves that $f$ is a continuity point for $\Delta_{\phi}(f)$. Hence the generic set $\bigcap _{\alpha \in \mathbb{Q}}U_{\alpha}$ is the set of continuity points of $\Delta_{\phi}(f)$.
  
\end{proof}

\begin{lemma}\label{pressure}
  For any continuous function $\phi:M\mapsto \mathbb{R}$, there exists a residual subset $\mathcal{G}$ of $\mathcal{E}_{\omega} (M)$ such that for any $f\in \mathcal{G}$ we have
  \[ \Delta_{\phi} (f)\leq \sup_{\tmop{Horseshoe}\ K} P (f|_K, \phi |_K).\]
\end{lemma}

\begin{proof}
  Let $\mathcal{G}_1$ be the set of continuity points of $\Delta_{\phi}$ over $\mathcal{E}_{\omega} (M)$. 
  
  \ 
  
  Let $\mathcal{G}_2$ be the set of diffeomorphisms $f$ admitting a hyperbolic periodic orbit whose homoclinic class $H(O)$ is the whole manifold $M$. From \cite{BONATTI2003839,arnaud_bonatti_crovisier_2005} this set contains a generic subset of $\mathcal{E}_{\omega} (M)$. 
  
  \ 
  
  Consider $\mathcal{G}_3$ the set of all diffeomorphisms $f$ such that for any hyperbolic periodic orbit $O$ of $f$ and any $\epsilon>0$, there exists a periodic orbit $O^{\prime}$ which is $\epsilon$-dense in $H(O)$ and whose collection of Lyapunov exponents is $\epsilon$-close to the Lyapunov exponents of $O$. Hyperbolic periodic points are dense in $H(O)$ by definition. We can then pick finitely many hyperbolic periodic points that are $\epsilon$-dense in $H(O)$. Then there is a horseshoe containing $O$ that is $\epsilon$-dense in $H(O)$. By Proposition \ref{p.control-exponents} we know that $\mathcal{G}_3$ is dense in $\mathcal{G}_2$. 
  
  \ 
  
  Let $f \in \mathcal{G}_1 \cap \mathcal{G}_2 \cap \mathcal{G}_3 \cap \mathcal{E}_{\omega} (M)$. This is a dense subset in $\mathcal{E}_{\omega} (M)$. We then fix a neighborhood $\mathcal{U}$ of $f$, $N,T$ as given by the main theorem. Let $\epsilon>0$ be small, we can find a periodic orbit $O$ such that 
  \begin{equation*}
  	 \Delta_{\phi}(f,O)>\Delta_{\phi}(f)-\epsilon/4.
  \end{equation*}
   Since $f \in \mathcal{G}_2 \cap \mathcal{G}_3$, we can replace $O$ by a periodic orbit $\mathcal{O}(p)$ which is $\epsilon$-dense in $H(O)=M$. This implies that $p$ has a large period, in particular we can chose $p$ such that its period is larger than $T$. Since $f \in \mathcal{E}_{\omega}(M)$, for any $N$ and invariant splitting $TM=E\oplus F$, there are $ x\in M, n \geq N, \eta>\frac{1}{2}$, such that for some unit vectors $u\in E(x), v\in F(x)$, we have the following,
  \[\Vert D_x f^{n}u \Vert > \eta \Vert D_x f^{n} v \Vert.\]
  Since $\mathcal{O}(p)$ is $\epsilon$-dense in $H(O)=M$, there is an iteration of $p$, say $f^m(p)$ which is $\epsilon$-close to $x$. For any fixed $C \in (0,1)$, we can make $\epsilon$ small enough, such that
  
  \[ C<\frac{\Vert D_{f^m(p)} f^{n}u\Vert}{\Vert D_x f^{n} u\Vert}<C^{-1}, \]
  \[ C<\frac{\Vert D_{f^m(p)} f^{n}v\Vert}{\Vert D_x f^{n} v\Vert}<C^{-1}. \]
  
  Hence,
    \[ \frac{\Vert D_{f^m(p)} f^{n}u\Vert}{\Vert D_{f^m(p)} f^{n} v\Vert}=\frac{\Vert D_{f^m(p)} f^{n}u\Vert}{\Vert D_x f^{n} u\Vert}\cdot \frac{\Vert D_x f^{n} u\Vert}{\Vert D_x f^{n} v\Vert}\cdot\frac{\Vert D_x f^{n} v\Vert}{\Vert D_{f^m(p)} f^{n} v\Vert} > C\cdot \eta \cdot C.\]
    
    We fix $C$ close to $1$ such that $\eta C^2>\frac{1}{2}.$
    
    \ 
    
   This proves that $p$ is a $T,N$-weak periodic point. By Theorem~ \ref{main}, we can find $g$ in $\mathcal{U}$ having a horseshoe $K$ such that 
   \begin{eqnarray*}
   	     P (g|_K, \phi |_K) &>& \Delta_{\phi}(f,\mathcal{O}(p))-\epsilon/4 \\
   	     &>&\Delta_{\phi}(f)-\epsilon/2. 
   \end{eqnarray*}
   
   \ 
   
   By structural stability of horseshoe, this property is open. Since $f\in \mathcal{G}_1$, we have a non-empty open set of diffeomorphisms $h$ having a horseshoe $K$ such that
   
   \begin{eqnarray*}
   	    P(h|_K,\phi|_K)&\geq& \Delta_{\phi}(f)-\epsilon/2 \\
   	                                    &>&\Delta_{\phi}(g)-\epsilon/2-\epsilon/2 \\
   	                                    &=&\Delta_{\phi}(g)-\epsilon .
   \end{eqnarray*}
   Denote this set by $V_{f,\epsilon}$. Then $\bigcup _{f\in \mathcal{G}_1 \cap \mathcal{G}_2 \cap \mathcal{G}_3 \cap \mathcal{E}_{\omega} (M)}V_{f,\epsilon}$ is open and dense in $\mathcal{E}_{\omega} (M)$. Then
   \[\bigcap_{n=1}^{\infty}\left( \bigcup _{f\in \mathcal{G}_1 \cap \mathcal{G}_2 \cap \mathcal{G}_3 \cap \mathcal{E}_{\omega} (M)}V_{f,\frac{1}{n}} \right)\]
   is a dense $G_{\delta}$ subset of $ \mathcal{E}_{\omega} (M)$ on which $\Delta_{\phi} (f)\leq \sup_K P (f|_K, \phi |_K)$.
\end{proof}

\begin{proof}[Proof of Theorem \ref{PressureFormula}]
Clearly, we have $P (f|_K, \phi |_K) \leq P (f, \phi)$ for any horseshoe $K$.
Then, Lemma \ref{one} and Lemma \ref{pressure} will immediately implies that for a given $\phi$, there is a residual subset $\mathcal{G}_{\phi}$ of $\mathcal{E}_{\omega} (M)$, such that for any $f\in \mathcal{G}_{\phi}$

\begin{eqnarray*}
	\sup_{\tmop{Horseshoe}\ K} P (f|_K, \phi |_K) &\leq& P (f, \phi) \\
	                                                                                             &\leq& \Delta_{\phi} (f) \\
	                                                                                             &\leq&\sup_{\tmop{Horseshoe}\ K} P (f|_K, \phi |_K),
\end{eqnarray*}
which implies
\begin{equation}\label{equation1}
 P (f, \phi)= \Delta_{\phi} (f)= \sup_{\tmop{Horseshoe}\ K} P (f|_K, \phi |_K).
\end{equation}
The space $C^0(M,\mathbb{R})$ of real valued continuous functions on $M$, is separable (because $M$ is compact), thus there is a countable, dense subset $\{\varphi_1,\ \varphi_2, \dots, \ \varphi_n, \dots\}$ of $C^0(M,\mathbb{R})$. Let $\mathcal{G}_n$
be the residual subset of $\mathcal{E}_{\omega} (M)$ on which equation (\ref{equation1}) holds for $\varphi_{n}$, then $\mathcal{G}=\cap_{n=1}^{\infty}\mathcal{G}_n$ is a residual subset of $\mathcal{E}_{\omega} (M)$ on which equation (\ref{equation1}) holds for all $\varphi_n$. Since $P (f, \phi)$, $\Delta_{\phi} (f)$ and $\sup_K P (f|_K, \phi |_K)$ are all continuous with respect to $\phi$ and $\{\varphi_1,\ \varphi_2, \dots, \ \varphi_n, \dots\}$ is dense in $C^0(M,\mathbb{R})$, for any $f$ in $\mathcal{G}$ and any $\phi \in C^0(M,\mathbb{R})$, equation (\ref{equation1}) holds. This proves Theorem \ref{PressureFormula}. 
\end{proof}

\subsection{Continuity of topological pressure}\label{C}

\begin{proof}[Proof of Theorem \ref{continuity}]
 By Theorem \ref{PressureFormula}, there is a residual subset $\mathcal{G}$ of $\mathcal{E}_{\omega}(M)$ such that for any $f\in \mathcal{G}$ and any $\phi$ in $C^0(M,\mathbb{R})$, $P (f, \phi) =\sup_K P (f|_K, \phi |_K)=\Delta_{\phi}(f)$. Here $K$ is a horseshoe of $f$. Let $\{\varphi_1,\ \varphi_2, \dots, \ \varphi_n, \dots\}$ be dense in $C^0(M,\mathbb{R})$. By Lemma \ref{ContinuitPoints}, there is a sequence of residual subsets $\{ \mathcal{G}_n \}_{n=1}^{\infty}$ of $\mathcal{E}_{\omega}(M)$ such that each $f\in \mathcal{G}_n$ is a continuity point of the map $g \mapsto \Delta_{\varphi_n}(g)$. Then each $f\in \mathcal{G}^{\prime}=\cap_{n=1}^{\infty} \mathcal{G}_n$ is a continuity point for the maps $g \mapsto \Delta_{\varphi_n}(g)$ for all $n$. Clearly, $\mathcal{G}^{\prime}$ is a residual subset of $\mathcal{E}_{\omega}(M)$. For any $f\in \mathcal{G}^{\prime}$ and any $\phi\in C^0(M,\mathbb{R})$, take a sequence $\{ f_n\}\subset \tmop{Diff}^1_{\omega} (M)$ such that $f_n\rightarrow f$. For any $\epsilon>0$, let $\varphi_m$ be $\epsilon /3$ close to $\phi$ in the $C^0$-topology, then for any $g\in \tmop{Diff}^1_{\omega} (M)$
 
 \begin{eqnarray*}
 \left| \Delta_{\varphi_m}(g)-\Delta_{\phi}(g)\right|&=&\left|\sup_{p \in \tmop{Per} (g)}  \Delta_{\varphi_m}(g,p)-\sup_{p \in \tmop{Per} (g)}  \Delta_{\phi}(g,p) \right| \\
 &\leq&  \sup_{p \in \tmop{Per} (g)}  \left| \Delta_{\varphi_m}(g,p)-\Delta_{\phi}(g,p) \right| \\
 &=& \sup_{p \in \tmop{Per} (g)}  \frac{1}{T(p)} \left|  \sum_{i = 0}^{T (p) - 1} (\varphi_m (g^i (p))-\phi (g^i (p))) \right| \\
 &\leq& \sup_{p \in \tmop{Per} (g)}  \frac{1}{T(p)}  \sum_{i = 0}^{T (p) - 1} \left|\varphi_m (g^i (p))-\phi (g^i (p)) \right| \\
 &<& \frac{\epsilon}{3}.
 \end{eqnarray*}
 Therefore 
 \begin{eqnarray*}
   \left| \Delta_{\phi}(f_n)-\Delta_{\phi}(f)\right| &=&\left| \Delta_{\phi}(f_n)-\Delta_{\varphi_m}(f_n)+\Delta_{\varphi_m}(f_n)-\Delta_{\varphi_m}(f)+\Delta_{\varphi_m}(f)-\Delta_{\phi}(f) \right| \\
   &\leq& \left|\Delta_{\phi}(f_n)-\Delta_{\varphi_m}(f_n)\right|+\left|\Delta_{\varphi_m}(f_n)-\Delta_{\varphi_m}(f)\right|+\left|\Delta_{\varphi_m}(f)-\Delta_{\phi}(f) \right|\\
   &<& \left|\Delta_{\varphi_m}(f_n)-\Delta_{\varphi_m}(f)\right| +\frac{2\epsilon}{3}.
 \end{eqnarray*}
Since $f\in \mathcal{G}^{\prime}$, there exists $N\in \mathbb{N}$ such that for all $n>N$,
\begin{equation*}
	\left|\Delta_{\varphi_m}(f_n)-\Delta_{\varphi_m}(f)\right|\leq \frac{\epsilon}{3}. 
\end{equation*}
Hence for all $n>N$,
\begin{equation*}
	\left| \Delta_{\phi}(f_n)-\Delta_{\phi}(f)\right|<\epsilon. 
\end{equation*}
This shows that each $f\in \mathcal{G}^{\prime}$ is a continuity point of the maps $g \mapsto \Delta_{\phi}(g)$ for any $\phi\in C^0(M,\mathbb{R})$. Then for any $f\in \mathcal{G}^{\prime\prime}:=\mathcal{G}\cap\mathcal{G}^{\prime}$ and any $\phi\in C^0(M,\mathbb{R})$, we have 
\begin{itemize}
	\item $P (f, \phi)= \Delta_{\phi} (f)= \sup_K P (f|_K, \phi |_K).$ 
	\item $f$ is a continuity point of the maps $g \mapsto \Delta_{\phi}(g)$, $g\in \tmop{Diff}^1_{\omega} (M)$.
\end{itemize}
By the structural stability of horseshoe and the continuity of the potential function, each $f\in \mathcal{G}^{\prime\prime}$ is a lower semi-continuity point for the map $f: \mapsto P (f, \phi) $ for each fixed $\phi$. 

\ 

For any $\alpha>0$, any $f\in \mathcal{G}^{\prime\prime}$, any $\widetilde{f}\in \tmop{Diff}^1_{\omega} (M)$ $C^1$-close to $f$ and $\mu$ be an ergodic measure of $\widetilde{f}$ such that 
\begin{equation*}
	h_{\mu}(\widetilde{f}) + \int \phi d \mu >P(\widetilde{f},\phi)-\frac{\alpha}{4}. 
\end{equation*}

We take $\widetilde{f}$ to be so close to $f$ such that 
\begin{equation*}
	\Delta_{\phi}(f)+\frac{\alpha}{4} > \Delta_{\phi}(\widetilde{f})
\end{equation*}
 By Theorem \ref{t.ergodic-closing}, there is a $C^1$-perturbation $\widehat{f} \in \tmop{Diff}^1_{\omega} (M)$ of $\widetilde{f}$ and a hyperbolic periodic point $p$ of $\widehat{f}$ such that 
 \begin{equation*}
 	\Delta_{\phi}(f)+\frac{\alpha}{2} > \Delta_{\phi}(\widehat{f})
 \end{equation*}
  and 
  \begin{equation*}
  	 \Delta (\widehat{f}, p) +
  	 \frac{1}{T (p)}  \sum_{i = 0}^{T (p) - 1} \phi (\widehat{f}^i (p))>\Delta(\widetilde{f},\mu) + \int \phi d \mu-\frac{\alpha}{4}.
  \end{equation*}
   Then we have 
\begin{eqnarray*}
P(f,\phi)&=&\Delta_{\phi}(f) \\
            &>&\Delta_{\phi}(\widehat{f})-\frac{\alpha}{2} \\
            &\geq& \Delta (\widehat{f}, p) + \frac{1}{T (p)}  \sum_{i = 0}^{T (p) - 1} \phi (\widehat{f}^i (p))-\frac{\alpha}{2} \\
            &>& \Delta(\widetilde{f},\mu) + \int \phi d \mu-\frac{3\alpha}{4} \\
            &\geq& h_{\mu}(\widetilde{f})+ \int \phi d \mu-\frac{3\alpha}{4} \\
            &>& P(\widetilde{f},\phi)-\alpha.
\end{eqnarray*}
Notice that Rulle's inequality $h_{\mu}(\widetilde{f})\leq \Delta(\widetilde{f},\mu)$ was used to obtain the implication from the fourth to the fifth inequality
. This shows that each $f\in \mathcal{G}$ is a upper semi-continuity point for the map $f: \mapsto P (f, \phi) $ for each fixed $\phi$. Upper semi-continuity and lower semi-continuity together imply continuity and Theorem \ref{continuity} is proved. 
\end{proof}
 
\begin{comment}
 Since the set of continuity points of a semi-continuous function on a Baire space always contains a residual subset, for each $\varphi_n$, there is a residual subset $\mathcal{G}_n\subset \mathcal{G}$ of $\mathcal{E}_{\omega} (M)$ whose elements are all continuity points of the map $f: \mapsto P (f, \varphi_n) $. Let $\mathcal{G}^{\prime}=\cap_{n=1}^{\infty} \mathcal{G}_n$, then $\mathcal{G}^{\prime}$ is a residual subset of $\mathcal{E}_{\omega} (M)$ such that each point in it is a continuity point for the maps $f: \mapsto P (f, \varphi_n) $ for all $n \in \mathbb{N}$. 

We can also use the same argument as in \cite[Corollary 1.1]{2016arXiv160601765B} to make the residual set in Theorem \ref{PressureFormula}  satisfy the conclusion of this theorem.
\end{comment}

\subsection{Another formula for topological pressure}\label{Jac}
\ 

Recall that for any $C^1$-diffeomorphism $f$ of a compact, $d_0$ dimensional, Riemannian manifold $M$ and for each integer $1\leq k\leq d_0$, 
\[\sigma_k(f,\phi):=\lim_{n \rightarrow \infty} \sup_{E\in \tmop{Grass}_k(TM)}\frac{1}{n}\left( \log \vert \tmop{Jac}(f^n,E)\vert + S_n \phi(\pi_k(E))\right),\]
where $\pi_k(E)$ is the base-point of $E$ and $S_n(\phi)$ is the Birkhoff sum of $\phi$.

\begin{lemma} \label{JacLemma}For $f\in \tmop{Diff}^1_{\omega}(M)$, we have the following,
    \begin{enumerate}
    	\item The limit defining $\sigma_k(f,\phi)$ exists and 
    \[ \sigma_k(f,\phi)= \inf_{n\geq 1}\sup_{E\in \tmop{Grass}_k(TM)}\frac{1}{n}\left( \log \vert \tmop{Jac}(f^n,E)\vert + S_n \phi(\pi_k(E))\right). \]
    	\item For each $k$, the map $f \mapsto  \sigma_k(f,\phi)$ is upper semi-continuous in the $C^1$ topology. More precisely, for any $f$ and $\alpha>0$, there exists $N_1$, such that for every $1\leq k \leq d_0$ and $g$ $C^1$-close to $f$ we have,
    	\[\forall n \geq N_1, \forall E\in \tmop{Grass}_k(TM), \log \vert \tmop{Jac} (g^n,E)\vert +S_n \phi(\pi_k(E)) \leq (\sigma_k(f,\phi)+\alpha)n\]
    \end{enumerate}
\end{lemma}
\begin{proof}
For the first item, note that,
\[ \left(\sup_{E\in \tmop{Grass}_k(TM)}\left( \log \vert \tmop{Jac}(f^n,E)\vert + S_n \phi(\pi_k(E))\right) \right)_{n\geq 1}\]
is subadditive. Hence, we have the convergence and the first equality of the first item.

 The second item is also a consequence of the subadditivity of 
 \[\left(\sup_{E\in \tmop{Grass}_k(TM)}\left( \log \vert \tmop{Jac}(f^n,E)\vert + S_n \phi(\pi_k(E))\right) \right)_{n\geq 1}\]
 
 and continuity of the map $(g,E)\mapsto \log \vert \tmop{Jac}(g,E)\vert + S_n \phi(\pi_k(E))$. 
\end{proof}

For $\mu \in \mathbb{P}_{\tmop{erg}}(f)$, we define
\[\sigma_k(f,\phi,\mu):=\sum_{j=d_0-k+1}^{d_0}\lambda_j(f,\mu)+\int \phi \ d\mu\]

\begin{lemma}\label{lemma1}
	For any $1\leq k \leq d_0$, we have 
	\[ \sigma_k(f,\phi)=\sup_{\mu \in \mathbb{P}_{\tmop{erg}}(f)}\sigma_k(f,\phi,\mu)  \]
\end{lemma}

	The proof of this Lemma is essentially the same as that of ~\cite[Lemma 5.7]{2016arXiv160601765B}. We give a proof for completeness. 
\begin{proof}
	Let $\mu \in \mathbb{P}_{\tmop{erg}}(f)$. By Oseledets theorem \cite{Osedelec}, we know that for $\mu$-almost every $x\in M$, there is a $k$-dimensional subspace of $T_xM$ whose  volume grows at the rate given by the $k$ largest Lyapunov exponents of $\mu$. Therefore, 
	
	\begin{eqnarray*}
		\sigma_k(f,\phi,\mu)&=& \inf_{n \geq 1} \sup_{E\in \tmop{Grass}_k(T_xM)} \frac{1}{n} \log \vert \tmop{Jac}(f^n,E)\vert + \int \phi \ d\mu \\
		&\leq& \sigma_k(f,\phi).
	\end{eqnarray*}

	The inequality is because the supremum in $\sigma_k(f,\phi,\mu)$ is taken over $\tmop{Grass}_k(T_xM)$ while that of $\sigma_k(f,\phi)$ is taken over $\tmop{Grass}_k(TM)$. 
	
	For the converse inequality, we consider the homeomorphism $f_k$ on $\tmop{Grass}_k(TM)$ defined by \[f_k(E)=Df(E)  \] 
	and the continuous function $\psi$ defined by \[ \psi_k(E) = \log \vert \tmop{Jac}(f,E)\vert + \phi(\pi_k(E)).  \]
	
	\begin{eqnarray*}
		\sigma_k(f,\phi)&=&\lim_{n \rightarrow \infty}\sup_{E\in \tmop{Grass}_k(TM)} \left(\frac{1}{n}\sum_{i=0}^{n-1}\psi_k(f_k^i(E))\right) \\
		&=&\sup_{\nu \in \mathbb{P}_{\tmop{erg}} (f_k)} \int_{\tmop{Grass}_k(TM)}\psi_k\ d\nu.
	\end{eqnarray*}

	The second equality is a well known result, for a proof, see \cite[Proposition 3.1]{Inoquio-Renteria2012}.
	Let $\nu \in \mathbb{P}_{\tmop{erg}} (f_k)$. Then the measure $\mu= (\pi_k)_{*}\nu$ is invariant under $f$. Birkhoff Ergodic Theorem gives $x\in M$ satisfying Oseledets theorem and $E_0 \in \tmop{Grass}_k(T_xM)$ such that 
	\begin{eqnarray*}
	\int_{\tmop{Grass}_k(TM)}\psi_k\ d\nu&=& \lim_{n \rightarrow \infty}\frac{1}{n} \log \vert \tmop{Jac}(f^n,E_0) \vert + \int \phi \ d\mu\\
	&\leq& \lim_{n \rightarrow \infty} \sup_{E\in \tmop{Grass}_k(T_xM)}\frac{1}{n}\log \vert \tmop{Jac}(f^n,E)\vert + \int \phi \ d\mu \\
	&=&\sigma_k(f,\phi,\mu).
	\end{eqnarray*}
	Hence, $\sigma_k(f,\phi)\leq\sup_{\mu \in \mathbb{P}_{\tmop{erg}}(f)}\sigma_k(f,\phi,\mu)$.

\end{proof}

\begin{lemma}\label{lemma2}
      There exists a residual subset $\mathcal{G}$ of $\mathcal{E}_{\omega} (M)$ such that for any $f\in \mathcal{G}$ and any $\phi \in C^0(M,\mathbb{R})$ we have
      \[ \Delta_{\phi}(f)=\max_{1\leq k\leq d_0}\sigma_k(f,\phi). \]
\end{lemma}

		This is a trivial generalization of ~\cite[Proposition 5.8]{2016arXiv160601765B}. We point out that the residual subset $\mathcal{G}$ of $\mathcal{E}_{\omega} (M)$ does not depend on the potential function $\phi$ because it is given by Corollary \ref{c.ergodic-closing}. 
\begin{proof}
	We know 
      \[ \Delta_{\phi}(f)=\sup_{p \in \tmop{Per} (f)}\Delta_{\phi}(f,p)=\sup_{p \in \tmop{Per} (f)}\max_{1\leq k\leq d_0}\sigma_k(f,\phi,p). \]
      Let $f\in \tmop{Diff}^1_{\omega}(M)$ be generic as in Corollary \ref{c.ergodic-closing} . Then each $\mu \in \mathbb{P}_{\tmop{erg}}(f)$ is approximated by periodic points with arbitrarily close Lyapunov exponents, so we have
      \[ \sup_{p \in \tmop{Per} (f)}\sigma_k(f,\phi,p)\geq \sup_{\mu \in \mathbb{P}_{\tmop{erg}}(f)}\sigma_k(f,\phi,\mu).  \]
      Since every periodic orbit gives an ergodic measure, we actually have equality for the above inequality. 
      Apply Lemma \ref{lemma1}, we have 
      \[ \Delta_{\phi}(f)=\max_{1\leq k\leq d_0}\sup_{\mu \in \mathbb{P}_{\tmop{erg}}(f)}\sigma_k(f,\phi,\mu)=\max_{1\leq k\leq d_0}\sigma_k(f,\phi). \]
\end{proof}
By Lemma \ref{lemma1} and Lemma \ref{lemma2}, we have proved Theorem \ref{Pressureformula}.

\section{No Equilibrium states with positive entropy}
In this section, we will prove that for generic diffeomorphism in $\mathcal{E}_{\omega} (M)$ and any continuous potential function, there are no equilibrium states with positive measure theoretic entropy.

Let $\mathcal{R}\subset \mathcal{E}_{\omega}(M)$ be a dense $G_{\delta}$ subset on which for every $\phi\in C^0(M,\mathbb{R})$
\begin{description}
	\item[1] $f\mapsto P (f, \phi)$ is continuous and
	\item[2] $P (f, \phi)=\max_{1\leq k\leq d_0}\sigma_k(f,\phi)=\Delta_{\phi}(f)$.
\end{description}

The following perturbation lemma will be the key to prove the non-existence of equilibrium states with positive entropy.
\begin{lemma}\label{PerturbationLemma}
	For any $g\in \mathcal{R}$, any continuous potential function $\phi$, any $\eta \in (0,\eta_0(g))$, any $\delta>0$, and any periodic saddle $O$ for $g$, there exists $N_0\geq 0$ with the following property.
	
	For any $\rho >0$, there is an $f_0 \in \mathcal{R}$ that is $\eta$-close to $g$ and satisfies:
	\begin{description}
		\item[1] $f_0$ can be picked such that it is arbitrarily $C^1$-close to $g$ outside the $\rho$-neighborhood of $O$;
		\item[2] $P (f_0, \phi)\geq \Delta_{\phi}(g,O)+\frac{\eta}{10 C_0^2\left( \Vert Dg\Vert+ \Vert Dg^{-1}\Vert \right)}$, where $\Vert Dg\Vert:=\sup_{x\in M} \Vert Dg(x)\Vert$, $\Vert Dg^{-1}\Vert:=\sup_{x\in M} \Vert Dg^{-1}(x)\Vert$; and 
		\item[3] for any $1\leq k\leq d_0$, any $E\in  \tmop{Grass}_k(TM)$, and any $n\geq N_0$,
		\[  \log \vert \tmop{Jac} (f_0^n,E)\vert +S_n \phi(\pi_k(E)) \leq (P(f_0,\phi)+\delta)n.\]
	\end{description}
\end{lemma}
\begin{remark}
	This is very similar to the construction in \cite[Lemma 6.2]{2016arXiv160601765B}. The difference is the choice of $N_2$ in the proof. Also notice that the difference between item 3 of this theorem and Lemma \ref{JacLemma} is that $N_0$ here does not depend on $f_0$ and $\rho$. 
\end{remark}
\begin{proof}
	Let $\eta_0(g)>0$ be small enough such that for any $f$ that is $\eta_0(g)$-close to $g$ in $C^1$-distance, we have $\sup_{E\in \tmop{Grass}_k(TM)}\vert  \tmop{Jac} (f,E) \vert \leq 2 \sup_{E\in \tmop{Grass}_k(TM)} \vert  \tmop{Jac} (g,E) \vert$. The construction of the first two items are the same as that of ~\cite[Lemma 6.2]{2016arXiv160601765B}. For completeness, we include the constructions here.
	\
	
	Let 
	\[ b=\frac{\eta}{8C_0^2 (\Vert Dg \Vert+\Vert Dg^{-1} \Vert)} \]
	For each point $x\in O$, there exist a linear map $C: T_xM \rightarrow \mathbb{R}^{d_0}$ which sends the stable subspace $E^s(x)$ to $\mathbb{R}^{\dim E^s(x)}\times \lbrace 0 \rbrace^{d_0-\dim E^s(x)}$ and satisfies $\Vert C \Vert, \Vert C^{-1} \Vert \leq C_0$. Define
	\[ U_x=C^{-1} \circ \exp 
	\begin{pmatrix}
	 \frac{-b}{\dim (E^s)} I_{\dim (E^s)} & 0 \\
	 0 & \frac{b}{\dim (E^u)} I_{\dim (E^u)}
	\end{pmatrix} 
	\circ C \]
	Note that 
	\begin{eqnarray*}
		\Vert U_x \circ Dg(x)-Dg(x) \Vert &\leq&  2b C_0^2 \max (\Vert Dg \Vert,\Vert Dg^{-1} \Vert), \\
		\Vert (U_x \circ Dg(x))^{-1}-Dg(x)^{-1} \Vert &\leq& 2b C_0^2 \max (\Vert Dg \Vert,\Vert Dg^{-1} \Vert). 
	\end{eqnarray*}
	Let $V$ be the $\rho$-neighborhood of $O$. By Franks' Lemma (Theorem \ref{t.linearize}), there is a $(\eta/2, V, O)$-perturbation $f^{\prime}$ with $Df^{\prime}(x)=U_x \circ Dg(x)$ for all $x\in O$. $O$ is also a hyperbolic periodic orbit for $f^{\prime}$ and $\Delta_{\phi}(f^{\prime},O)=\Delta_{\phi}(g,O)+b$. Choose $f_0\in \mathcal{R}$ arbitrarily $C^1$-close to $f^{\prime}$, so it is arbitrarily $C^1$-close to $g$ outside the $\rho$-neighborhood of $O$. By the structural stability of hyperbolic periodic orbit, there is a hyperbolic periodic orbit $O^{\prime}$ for $f_0$ close to $O$ in Haursdorff distance such that,
	\begin{eqnarray*}
		P(f_0,\phi)&=&\Delta_{\phi}(f_0) \\
		                     &\geq& \Delta_{\phi}(f_0,O^{\prime}) \\
		                     &>& \Delta_{\phi}(f^{\prime},O)-b/5 \\
		                     &=&\Delta_{\phi}(g,O)+4b/5 		                     
	\end{eqnarray*}
     This gives the first two items of the lemma. 
	
	In order to prove the third item, we first define the following for each $0\leq k\leq d_0$, 
	\begin{description}
		\item[1] An integer $N_1=N_1(g,\delta)$ such that for any diffeomorphism $f$ $C^1$-close to $g$, any $E\in \tmop{Grass}_k(TM)$, and any $n\geq N_1$, 
		\[ \log \vert \tmop{Jac} (g^n,E)\vert +S_n \phi(\pi_k(E)) \leq (\sigma_k(f,\phi)+\delta/3)n. \]
		\item[2] A diffeomorphism $G$ on the tangent bundle $T_OM$, defined as follow. Let $V_r$ be the union of the $r$-balls at the the origin in each $T_xM, x \in O$. Let $u=u_s+u_u \in T_xM=E^s(x)\oplus E^u(x)$ and choose the norm $\Vert \cdot \Vert$ on $T_xM$ to be $\Vert u \Vert := \Vert u_s\Vert_s + \Vert u_u \Vert_u$, where $\Vert \cdot \Vert_s$ and $\Vert \cdot \Vert_u$ are the norm induced by Riemannian metric on $E^s(x)$ and $E^u(x)$ respectively. Notice that in this norm, $\Vert u_s\Vert=\Vert u_s\Vert_s$ and $\Vert u_u\Vert=\Vert u_u\Vert_u$. Then we can find a diffeomorphism $G$ of $T_OM$ and $0<r_1<1<r_2<\infty$ such that 
		\begin{itemize}
			\item $G$ coincides with $Dg(x)$ outside the unit balls of each space $T_xM, x \in O$, and with $U_x \circ Dg(x)$ on $V_{r_1}\cap T_xM$, 
			\item $G$ is $\eta/2$-close to $Dg|_O$, and 
			\item if $\Vert u \Vert \geq r_2$, then $\Vert G^n(u) \Vert \geq 1$ for either all $n\geq 0$ or all $n\leq 0$. 
		\end{itemize}
	The first two items follow from Franks' Lemma(Theorem \ref{t.linearize}), and the third item requires a little bit explanation. Note that $O$ is a saddle for $Dg|_O$, then let $C>0, \lambda\in (0,1)$ be the constants in the definition of hyperbolic set. Let $\Vert u \Vert \geq r_2$. Then we have
	\begin{eqnarray*}
		\Vert G^n u \Vert &=& \Vert G^n (u_s+u_u) \Vert \\
		                               &\geq & C\lambda^{-n} \Vert u_u \Vert -C \lambda^n \Vert u_s \Vert\\
		                               &=& C\lambda^{-n} \Vert u \Vert \left(\frac{\Vert u_u \Vert}{\Vert u \Vert}- \frac{\lambda^{2n}\Vert u_s \Vert}{\Vert u \Vert}\right)\\
		                               &=&C\lambda^{-n} \Vert u \Vert \left(1- \frac{(1+\lambda^{2n})\Vert u_s \Vert}{\Vert u \Vert}\right).
	\end{eqnarray*}
For $n=0$, it is obvious that $\Vert G^n u \Vert =\Vert  u \Vert \geq r_2>1$. For any $n\geq 1$, we want 
\[ C\lambda^{-n} \Vert u \Vert \left(1- \frac{(1+\lambda^{2n})\Vert u_s \Vert}{\Vert u \Vert}\right)\geq 1. \]
Note the above expression attains its minimum at $n=1$, so we need 
\[ C\lambda^{-1} r_2 \left(1- \frac{(1+\lambda^{2})\Vert u_s \Vert}{\Vert u \Vert}\right)\geq 1, \]
or equivalently, 
\[ \frac{\Vert u_s \Vert}{\Vert u \Vert}\leq \frac{1-\lambda/(Cr_2)}{1+\lambda^{2}}. \]
A similar argument shows that for any $n\leq 0$,  we need 
\[ \frac{\Vert u_u \Vert}{\Vert u \Vert}\leq \frac{1-\lambda/(Cr_2)}{1+\lambda^{2}}. \]
So in order to make item 3 possible, we need at least one of the previous two inequalities always be satisfied. Recall $\Vert u \Vert =\Vert u_s\Vert + \Vert u_u \Vert$, therefore we need 
\[ \frac{1-\lambda/(Cr_2)}{1+\lambda^{2}}\geq 1/2. \]
 That is
\[ r_2 \geq \frac{2\lambda}{C(1-\lambda^2)}. \] 

Let $\Lambda$ be the maximal invariant set of $G$ in $V_{r_2}$. Let $ y\in V_{r_2}$, if $O(y)\subset V_{r_2}$, then $O(y)\subset \Lambda$. Hence the item 3 in the above construction of $G$ implies that $\sigma_k(G,\phi)$ is well defined and satisfies
\[ \sigma_k(G,\phi)=\max(\sigma_k(G|\Lambda,\phi), \sigma_k(Dg|_O,\phi))=\sigma_k(G|\Lambda,\phi). \]
The second equality is because $\sigma_k(U_x\circ Dg|_O,\phi)$ is larger than $\sigma_k(Dg|_O,\phi)$ and $O\subset \Lambda$.
\item[3] An integer $N_2$ such that for any $E\in \tmop{Grass}_k(TM)$ and $n\geq N_2$,
	\[ \log \vert \tmop{Jac} (G^n,E)\vert \leq (\sigma_k(G)+\delta/9)n \]
	and 
	\[ \left|  \frac{1}{n}S_n\phi(x)- \frac{1}{ \#O}  \sum_{i = 0}^{\#O - 1} \phi (g^i (x))\right|  <\delta/9 ,\forall x\in O.\]
	\end{description}
\subsection*{Construction of $f_0$}
 
 \
 
 Franks' Lemma (Theorem \ref{t.linearize}) gives a diffeomorphism $g^{\prime}$ which is linear in a small neighborhood of $O$ and arbitrarily close to $g$. The dynamics near $O$ for $g^{\prime}$ can be identified with the linear cocycle $Dg$ over the tangent bundle $T_OM$ in an obvious way. We choose $R>0$ to be small and define a perturbation $h$ of $g$ by replacing $g^{\prime}$ in a small neighborhood of $O$ by the diffeomorphism $G_R: z\mapsto RG(R^{-1}z)$.  Note that we identified the tangent bundle $T_OM$ with $M$ in a small neighborhood of $O$ and the neighborhood where $h$ equals $G_R$ is contained in the neighborhood on which $g^{\prime}$ is linear. In particular, $h$ preserves the set $\Lambda_R:=R\Lambda$ and $\sigma_k(h|\Lambda_R)=\sigma_k(G|\Lambda)$. Also, the Jacobian is not changed. Hence, $G_R$ satisfies the same inequality with the same $N_2$ as $G$. If we take $R$ to be so small (this means $\Lambda_R$ is small) that the variation of the potential function $\phi$ near each point in the periodic orbit $O$ is smaller than $\delta/18$. Then for $n\geq N_2$,
\[ \log \vert \tmop{Jac} (G_R|\Lambda_R^n,E)+S_n\phi(\pi_k(E))\vert \leq (\sigma_k(G_R|\Lambda_R)+\frac{1}{n}S_n\phi(x)+\delta/6)n\]
for the $x\in O$ that is closest to $\pi_k(E)$. 
Note that 
\[ \sigma_k(G_R|\Lambda_R,\phi)>\sigma_k(G_R|\Lambda_R) +\frac{1}{ \#O}  \sum_{i = 0}^{\#O - 1} \phi (g^i (x))-\delta/18.\]
Hence by definition of $N_2$, 
\[ \log \vert \tmop{Jac} (G_R|\Lambda_R^n,E)+S_n\phi(\pi_k(E))\vert \leq (\sigma_k(G_R|\Lambda_R,\phi)+\delta/3)n. \]
Let $N=\max(N_1,N_2)$. We can choose $R$ to be arbitrarily small once $N$ is given. It is then clear that any piece of orbit of $h$ of length $N$ coincide with a piece of orbit of $g^{\prime}$ or of $G_R$. Hence, for any $E \in \tmop{Grass}_k(T_xM)$,

\begin{eqnarray*}
	\log \vert \tmop{Jac}(h^N,E) \vert +S_N\phi(\pi_k(E))&<&(\max (\sigma_k(G_R|\Lambda_R,\phi),\sigma_k(g^{\prime},\phi))+\delta/3)N \\
	&\leq& (\sigma_k(h,\phi)+\delta/3)N.
\end{eqnarray*}

Note that $N$ only depends on $g$ and $G$. (This is very important for the sequel
of the argument.)
For each $k$, choose $\mu \in \mathbb{P}_{erg}(h)$ with $\sigma_k(h,\phi,\mu)>\sigma_k(h,\phi)-\delta/12$. By $\tmop{Ma \tilde{n}  \acute{e}}$'s Ergodic Closing lemma(Corollary \ref{c.ergodic-closing}), there is a perturbation $h^{\prime}$ with a periodic point $p$ such that $\sigma_k(h^{\prime},\phi,p)>\sigma_k(h,\phi,\mu)-\delta/12$. Franks' Lemma (Theorem \ref{t.linearize}) gives a further perturbation $h^{\prime \prime}$ with $p$ a periodic saddle of it and $\sigma_k(h^{\prime \prime},\phi,p)>\sigma_k(h^{\prime },\phi,p)-\delta/12$. By the stability of hyperbolic periodic point $p$, for any $f_0$ $C^1$-close to $h^{\prime \prime}$, there is a periodic saddle $p^{\prime }$ for $f_0$ close to $p$ such that $\sigma_k(f_0,\phi,p^{\prime})>\sigma_k(h^{\prime \prime},\phi,p)-\delta/12$.
Hence, we can choose $f_0 \in \mathcal{R}$ which also satisfies the first two items of this Lemma, thus 

\begin{eqnarray*}
	P(f_0,\phi)&\geq& \sigma_k(f_0,\phi) \\
	                     &>&\sigma_k(h,\phi)-\delta/3.
\end{eqnarray*}

Each of the above perturbation is arbitrarily small and $f_0$ is $\eta$-close to $g$.
Then for any $E\in \tmop{Grass}_k(TM)$, we have
\begin{eqnarray*}
	\log \vert \tmop{Jac}(f_0^N,E) \vert +S_N\phi(\pi_k(E))&<&\log \vert \tmop{Jac}(h^N,E) \vert +S_N\phi(\pi_k(E))+N \delta/12\\
	                                                                                             &\leq& (\sigma_k(h,\phi)+5\delta/12)N\\
	                                                                                             &\leq& (P(f_0,\phi)+3\delta/4)N.
\end{eqnarray*}
For $0\leq r<N$ and $l>l_0=[4(\Vert \phi \Vert_{\infty}+\sup_{E\in \tmop{Grass}_k(TM)} \log \vert 2\tmop{Jac}(g,E) \vert)/\delta]$, this implies
\begin{eqnarray*}
	&\log& \vert \tmop{Jac}(f_0^{lN+r},E) \vert +S_{lN+r}\phi(\pi_k(E))\\
	&<&(P(f_0,\phi)+3\delta/4)lN +(N-1)(\Vert \phi \Vert_{\infty}+\sup_{E\in \tmop{Grass}_k(TM)} \log \vert \tmop{Jac}(f_0,E) \vert) \\
	&<&(P(f_0,\phi)+3\delta/4)lN +N(\Vert \phi \Vert_{\infty}+\sup_{E\in \tmop{Grass}_k(TM)} \log \vert 2\tmop{Jac}(g,E) \vert)\\
	&<& (P(f_0,\phi)+\delta)lN\\
	&\leq&(P(f_0,\phi)+\delta)(lN+r).
\end{eqnarray*}
Let $N_0=(l_0+1)N$. Note that $N_0$ only depends on $N$,$\phi$,$\delta$, $g$ and the Riemannian metric.
\end{proof}
\begin{proposition}\label{non-generic-concentration}
	Fix $\epsilon,\alpha \in (0,1), \phi\in C^0(M,\mathbb{R})$. Then for any $f_0$ in a dense subset of $\mathcal{E}_{\omega}(M)$, there is a constant $\delta>0$ and a periodic orbit $O\subset M$ such that for any $f\in \mathcal{E}_{\omega}(M)$ $C^1$-close to $f_0$ and any ergodic measure $\mu$ for $f$, we have
	\[ h_{\mu}(f)+\int \phi \ d\mu>P(f,\phi)-\delta \ \textup{implies}\  \mu \left(M\setminus \bigcup_{x\in O} B_f(x,\epsilon,\#O)\right)<\alpha \]
\end{proposition}
\begin{remark}
	\textup{This is a generalization of ~\cite[Proposition 6.1]{2016arXiv160601765B}. Let $z$ be a point in the support of an ergodic measure that is Lyapunov regular and satisfies Birkhoff's Ergodic Theorem applied to the indicator function of $M\setminus \bigcup_{x\in O} B_f(x,\epsilon,\#O)$. The idea of the proof is to use the fact that the growth rate of the logarithm of Jacobian is larger when the iteration of the point $z$ is inside a small neighborhood of the periodic orbit because we can perturb the system near the periodic orbit by Lemma \ref{PerturbationLemma}. This will force the iteration of $z$ to stay close to the periodic orbit very often since $z$ is a point in the support of an ergodic measure with large pressure. In the proof, we will need Lemma \ref{PerturbationLemma} and the fact that $N_0$ only depends on $N$, $\phi$, $\delta$, $g$ and the Riemannian metric.}
\end{remark}

\begin{proof}

	We fix $\phi,\epsilon,\alpha$ and let $g\in \mathcal{R}$. Pick an arbitrarily small number $0<\eta\leq \eta_0(g)$ with $\eta_0(g)$ as give in the Lemma \ref{PerturbationLemma} and set 
	\begin{eqnarray}\label{delta}
	\delta=\frac{\alpha \eta}{100 C_0^2(\Vert Dg\Vert+ \Vert Dg^{-1}\Vert) }.
	\end{eqnarray}
	
	By Lemma \ref{JacLemma} and $g\in \mathcal{R}$, there is a $N_1$ such that for any $0\leq k\leq d_0$, any $E\in  \tmop{Grass}_k(TM)$, and any $n\geq N_1$, we have 
	
	\begin{equation}\label{ineq}
	\log \vert \tmop{Jac} (g^n,E)\vert +S_n \phi(\pi_k(E)) \leq (\sigma_k(g,\phi)+\delta)n\leq (P(g,\phi)+\delta)n.
	\end{equation}
	
	There exist a periodic orbit $O$ of $g\in \mathcal{R}$ such that 
	\begin{eqnarray}\label{periodic}
	\Delta_{\phi}(g,O)\geq P(g,\phi)-\delta .
	\end{eqnarray}
	
	Lemma \ref{PerturbationLemma} gives $N_0=N_0(g,O,\delta)$. Let $N=\max (N_0,N_1,\# O)$. For any $C^1$ diffeomorphism $h$ that is $2\rho$-close to $g$ in the $C^0$-distance, by the continuity of $h$ and the fact that $N_0$ does not depend on $\rho$, we can choose $\rho>0$ such that 
	\begin{equation}\label{inclusion}
	h^{-N}\left( \bigcup_{x\in O} B(x,\rho)\right) \subset \bigcup_{x\in O}B_g(x,\epsilon/2,\# O),
	\end{equation}
	and 
    \begin{equation}\label{hau}
    \bigcup_{x\in O}B_g(x,\epsilon/2,\# O) \subset \bigcup_{x\in O}B_h(x,\epsilon,\# O).
    \end{equation}
    Here $B(x,\rho)$ is the open ball centered at $x$ with radius $\rho$. 
    
    Item 3 of Lemma \ref{PerturbationLemma} gives a diffeomorphism $f_0$ that is $\eta$-close to $g$ in $\mathcal{E}_{\omega}(M)$ and $2\rho$-close to $g$ in the $C^0$-distance. For any $f\in \mathcal{E}_{\omega}(M)$ close to $f_0$ and any ergodic measure $\mu$ for $f$ such that 
    \begin{equation*}
    	h_{\mu}(f)+\int \phi \ d\mu>P(f,\phi)-\delta .
    \end{equation*}
     Notice that $\mu$ is also ergodic with respect to $f^N$. So we can estimate the time spent outside $B_f(x,\epsilon,\#O)$ by the forward orbit of $f^N$.
	
	As $f$ is close to $f_0$ and $f_0\in \mathcal{R}$ is a continuity point of $P(f,\phi)$,
	\begin{eqnarray*}
		h_{\mu}(f)+\int \phi \ d\mu&>&P(f,\phi)-\delta \\
		                                                       &>& P(f_0,\phi)-2\delta .
	\end{eqnarray*}

	By Ruelle's Inequality, there exists $0\leq k \leq d_0$ such that 
	\begin{equation*}
		\sigma_k(f,\phi,\mu)\geq h_{\mu}(f)+\int \phi \ d\mu .
	\end{equation*}
	 By Oseledets Theorem, for $\mu$-almost every $z$, there is $E\subset \tmop{Grass}_k(T_zM)$ with 
	\[ \lim_{n\rightarrow \infty} \left(\frac{1}{n}\log \vert \tmop{Jac} (f^n,E)\vert + \frac{1}{n}S_n \phi(\pi_k(E))\right) =  \sigma_k(f,\phi,\mu).\]
	
	And 
	\begin{eqnarray*}
		\sigma_k(f,\phi,\mu)&\geq& h_{\mu}(f)+\int \phi \ d\mu \\
		                                         &>& P(f_0,\phi)-2\delta
	\end{eqnarray*}	
	When $f^n(z)$ is in $\bigcup_{x\in O}B_f(x,\epsilon,\# O)$, item 3 of a Lemma \ref{PerturbationLemma} gives
	\[ \frac{1}{N}\log \vert \tmop{Jac} (f^N,Df^n(E))\vert +\frac{1}{N}S_N \phi(\pi_k(Df^n(E)))\leq P(f_0,\phi)+2\delta. \]
	When $f^n(z)$ is not in $\bigcup_{x\in O}B_f(x,\epsilon,\# O)$, the relations (\ref{inclusion}) (\ref{hau}) and item 1 of a Lemma \ref{PerturbationLemma} show that $f^N(f^n(z))$ and $g^N(f^n(z))$ are arbitrarily close. Hence, by inequality (\ref{ineq}), 
	\[ \frac{1}{N}\log \vert \tmop{Jac} (f^N,Df^n(E))\vert +\frac{1}{N}S_N \phi(\pi_k(Df^n(E)))\leq P(g,\phi)+2\delta. \]
	For each $m\geq 1$, we define 
	\[ p_m=\frac{1}{m}\#\lbrace 0\leq l<m : f^{lN}(z)\notin \bigcup_{x\in O} B_f(x,\epsilon,\#O)\rbrace.  \]
	By chain rule and the definition of Birkhoff sum,
	\[ \frac{1}{mN}\log \vert \tmop{Jac}(f^{mN},E)\vert+\frac{1}{mN}S_{mN} \phi(\pi_k(E))\leq (1-p_m)(P(f_0,\phi)+2\delta)+p_m(P(g,\phi)+2\delta). \]
	Let $m$ go to infinity, using Birkhoff Ergodic Theorem(for almost every $z$), (\ref{delta}), (\ref{periodic}) and item 2 of Lemma \ref{PerturbationLemma},
	
	\begin{eqnarray*}
		\mu \left(M\setminus \bigcup_{x\in O} B_f(x,\epsilon,\#O)\right)&\leq& \frac{4\delta}{P(f_0,\phi)-P(g,\phi)} \\
		&\leq& 100\delta C_0^2(\Vert Dg\Vert+ \Vert Dg^{-1}\Vert)/\eta \\
		&\leq& \alpha.
	\end{eqnarray*}
\end{proof}

\begin{proposition}\label{concentration}
	There is a dense $G_{\delta}$ set $\mathcal{G}\subset \mathcal{E}_{\omega}(M)$ such that for any $\epsilon,\alpha \in (0,1)$, any $f\in \mathcal{G}$ and $\phi\in C^0(M,\mathbb{R})$, there exist $\delta>0$ and a periodic orbit $O\subset M$ such that
	\[ h_{\mu}(f)+\int \phi \ d\mu>P(f,\phi)-\delta \ \textup{implies} \  \mu \left(M\setminus \bigcup_{x\in O} B_f(x,\epsilon,\#O)\right)<\alpha. \]
\end{proposition}
\begin{proof}
	Let $\{\varphi_1,\ \varphi_2, \dots, \ \varphi_n, \dots\}$ be dense in $C^0(M,\mathbb{R})$. Fix rational $\epsilon, \alpha$, let $f_0\in \mathcal{V}(\epsilon, \alpha,n)$ be the dense set in Proposition \ref{non-generic-concentration}, $V_{f_0,\epsilon,\alpha,n}$ be the open set of $f_0$ satisfying the conclusion of Proposition \ref{non-generic-concentration} with $\varphi_n$ the potential function. Then the union $\bigcup_{f_0\in \mathcal{V}(\epsilon, \alpha,n)} V_{f_0,\epsilon,\alpha,n}$ is open and dense in $\mathcal{E}_{\omega}(M)$. The intersection 
	\begin{equation*}
		\mathcal{G}:=\bigcap_{n=1}^{\infty}\bigcap_{\epsilon,\alpha\in \mathbb{Q}\cap (0,1)} \left(\bigcup_{f_0\in \mathcal{V}(\epsilon, \alpha,n)} V_{f_0,\epsilon,\alpha,n}\right)
	\end{equation*}
	 is a dense $G_{\delta}$-subset of $\mathcal{E}_{\omega}(M)$ satisfying the conclusion for $\varphi_n$ for every $n$. For any $f\in \mathcal{G}$ and $\phi\in C^0(M,\mathbb{R})$, if 
	\[ h_{\mu}(f)+\int \phi \ d\mu>P(f,\phi)-\delta, \]
	then because $\{\varphi_1,\ \varphi_2, \dots, \ \varphi_n, \dots\}$ is dense in $C^0(M,\mathbb{R})$, there is a sequence $\{ \varphi_{n_k} \}$ converges to $\phi $ in $C^0$-topology. Therefore by the continuity of $\phi \mapsto P(f,\phi)$ and $\phi \mapsto \int \phi \ d\mu$
	
	\begin{eqnarray*}
		\lim_{k \rightarrow \infty}(P(f,\varphi_{n_k})-\int \varphi_{n_k} \ d\mu) &=&P(f,\phi)-\int \phi \ d\mu \\
		&<&h_{\mu}(f)+\delta.
	\end{eqnarray*}

	Hence, for $k$ large enough, $P(f,\varphi_{n_k})-\int \varphi_{n_k} \ d\mu < h_{\mu}(f)+\delta$. This implies 
	\[\mu \left(M\setminus \bigcup_{x\in O} B_f(x,\epsilon,\#O)\right)<\alpha \]
	and the Proposition is proved.
\end{proof}
 
\begin{proof}[Proof of Theorem\ref{No-ES-positive-entropy}]
 This proof is essentially the same as the proof of ~\cite[Theorem 2]{2016arXiv160601765B}, we give a sketch of the argument. Let $f$ be in the dense $G_{\delta}$ set $\mathcal{G}\subset \mathcal{E}_{\omega}(M)$ as in Proposition \ref{concentration}. Assume that there is a equilibrium state $\mu \in \mathbb{P}_{\tmop{erg}}(f)$ for $(f,\phi)$ with positive entropy. By Katok's entropy formula \cite{PMIHES_1980__51__137_0}, we have 
	\[ h_{\mu}(f)=\lim_{\epsilon \rightarrow 0}\limsup_{n\rightarrow \infty}\frac{1}{n}\log r_f(\mu,\epsilon,n) \]
	where $r_f(\mu,\epsilon,n)$ is the minimal number of Bowen balls $B_f(x,\epsilon,n)$ needed to cover a set of $\mu$-measure greater than $1/2$ ($1/2$ can be replaced by any number between $0$ and $1$).
	
	Let us fix $\epsilon>0$ and some $\epsilon$-dense finite set $\mathcal{A} \subset M$. Let $0<\alpha\ll 1/\log (\# A)$. Proposition \ref{concentration} gives a number $\delta>0$ and a periodic orbit $O\subset M$ (where $\mathcal{A}$ is disjoint from $O$). Write $N=\#O$ for convenience. The fact that $h_{\mu}(f)>0$ implies that $N$ goes to infinity as $\epsilon$ goes to zero. If this is not true, $\mu$ will be an atomic measure by Proposition \ref{concentration}. This gives us zero measure entropy which contradicts our assumption.

Using the same estimates as in the proof of ~\cite[Theorem 2]{2016arXiv160601765B}, we have that 
\[ \frac{1}{n}\log r_f(\mu,2\epsilon,n)\leq \frac{\log n}{n}+H(1/N)+\alpha \log(\# \mathcal{A}) +\frac{1}{N}\log N \]
where $H(t)=-t\log t- (1-t)\log (1-t)$. $\mu$ is an equilibrium state implies that for any $\delta>0$, 

\[h_{\mu}(f)+\int \phi \ d\mu>P(f,\phi)-\delta \]

We can take $\epsilon$ and $\alpha \log (\# \mathcal{A})$ arbitrarily small. Hence,
\[h_{\mu}(f)=\lim_{\epsilon \rightarrow 0}\limsup_{n\rightarrow \infty}\frac{1}{n}\log r_f(\mu,\epsilon,n)=0\] 
which is a contradiction.
\end{proof}

\begin{comment}
\begin{remark}
	There is no essential difference between proof of Theorem \ref{No-ES-positive-entropy} and the proof of ~\cite[Theorem 2]{2016arXiv160601765B}. We add the proof here with more details for completeness. 
\end{remark}
\end{comment}

\section{Equilibrium states on compact surfaces with zero Lyapunov exponents} \label{zeroExponent}
\ 

We will restrict to the case of compact surface in this entire section. For an area-preserving diffeomorphism, if it has a dominated splitting on a compact invariant set, then the splitting is hyperbolic. So $\mathcal{E}_{\omega} (M)$ is the interior of the set of all non-Anosov diffeomorphisms on $M$.

Since $M$ is a compact surface, we have $\Delta (f,p)=\lambda^+(f,p)$.
By Theorem \ref{PressureFormula}, we know that
\[ P (f, \phi) = \sup_{p \in \tmop{Per} (f)} \{\lambda^+(f, p) + \frac{1}{T(p)}  \sum_{i = 0}^{T (p) - 1} \phi (f^i (p))\} \]
for generic diffeomorphism $f\in \mathcal{E}_{\omega} (M)$ and any continuous function $\phi$ on $M$. Using $\tmop{Ma \tilde{n}  \acute{e}}$'s ergodic closing lemma(Corollary \ref{c.ergodic-closing}), we also have the following
\[ P (f, \phi)= \sup_{\mu \in \mathbb{P}_{\tmop{erg}}(f)}\lbrace \lambda^+(f, \mu ) +\int \phi \ d\mu \rbrace. \]

By Theorem \ref{No-ES-positive-entropy}, if there is an equilibrium state, say $\nu$, then we must have 
\begin{equation*}
	\int \phi \ d\nu = P(f,\phi).
\end{equation*}
This implies that $\lambda^+(f, \nu )=0$. Conversely, if $\lambda^+(f, \nu )=0$, 
\begin{equation*}
	\int \phi \ d\nu= \sup_{\mu \in \mathbb{P}_{\tmop{erg}}(f)} \int \phi \ d\mu
\end{equation*}
and 
\begin{equation*}
	\lambda^+(f, \mu )+\int \phi \ d\mu \leq \int \phi \ d\nu
\end{equation*}
 for any ergodic measure $\mu$, then we have $\nu$ is an equilibrium state for $(f,\phi)$.

As mentioned in the introduction, for $\phi_m(x):=-\frac{1}{m}\log \Vert D_xf^m\Vert, m\in \mathbb{N}$, we have the following result. 
\begin{proposition}\label{GeometricPotential}
	There exists a residual subset $\mathcal{G}$ of $\mathcal{E}_{\omega} (M)$, such that for any $f\in \mathcal{G}$, there exists $m\in \mathbb{N}$ such that there exist an equilibrium state for $(f,\phi_m)$.
\end{proposition}

\begin{proof}
	By \cite{10.2307/2374000}, we can find an elliptic periodic point $p$ for $f$, let $m$ denote the period of $p$. 
	For any ergodic measure $\mu$, by Oseledets' Theorem \cite{Osedelec} and Birkhoff's Ergodic Theorem, we can find a set of full measure such that for any point $x$ in it we have
	\begin{eqnarray}
	\lambda^+(f, \mu )=\lambda_2&=&\lim_{n \rightarrow \infty}\frac{1}{n}\log \Vert D_xf^n \Vert, \\
	\lim_{k \rightarrow \infty} \frac{1}{k} \sum_{i=0}^{k-1} \phi_{m}(f^{im}(x))&=&\int \phi_{m} \ d \mu .
	\end{eqnarray}
	Let $n=km$ be a multiple of $m$, then 
	\begin{eqnarray}
	\frac{1}{n}\log \Vert D_xf^n \Vert &\leq&\frac{1}{km}\sum_{i=0}^{k-1}\log \Vert D_{f^{im}(x)}f^m \Vert \\
	&=& -\frac{1}{k}\sum_{i=0}^{k-1} \phi_{m}(f^{im}(x)) .
	\end{eqnarray}
	Letting $n$ go to infinity on both sides we have,
	\[\lambda^+(f, \mu )\leq -\int \phi_{m} \ d\mu. \]
	For the elliptic periodic point $p$, let $\mu_p$ be the corresponding atomic measure. It is clear that we have 
	\begin{equation*}
		\phi_{m}(p)=-\frac{1}{m}\log \Vert D_pf^m \Vert=0
	\end{equation*}
	 and $\lambda^+(f,\mu_p)=0$. This gives 
	 \begin{equation*}
	 	\lambda^+(f,\mu_p)+\int \phi_m \ d\mu_p=0.
	 \end{equation*}
	  Therefore for any ergodic measure $\mu$,
	\[ \lambda^+(f,\mu)+\int \phi_m \ d\mu\leq\lambda^+(f,\mu_p)+\int \phi_m \ d\mu_p=0. \]
	Hence $\mu_p$ is an equilibrium state for $(f,\phi_m)$.
\end{proof}

We can also consider the family of potential function $t\phi_{m}$ for $t\geq 0$. It is clear from Proposition \ref{GeometricPotential} that for $t\geq 1$, there exists at least one equilibrium state for $(f,t\phi_{m})$. And for $t=0$, we have that $P(f,0)=h_{\tmop{top}}(f)$, and there is no equilibrium states for $C^1$-generic conservative diffeomorphism without a dominated splitting by ~\cite[Theorem 2]{2016arXiv160601765B}. Since $P(f,t\phi_{m})$ is continuous with respect to $t$ \cite{walters2000introduction} and is monotone decreasing as $t$ increases, there is a phase transition point $t_0$. More specifically, we have the following. 

\begin{proposition}\label{transition}
	Let $f\in \mathcal{G}$ and $m\in \mathbb{N}$ as in Proposition \ref{GeometricPotential}, $\phi_{m}:=-\frac{1}{m}\log \Vert D_xf^m\Vert$.
	Let $t_0:=\min \lbrace t\ |\  P(f,t\phi_{m})=0 \rbrace$. Then 
	\[ t_0=\sup_{p \in \tmop{Per} (f)\cap S(f)}\frac{\lambda^+(f,p)}{-\frac{1}{T(p)}  \sum_{i = 0}^{T(p) - 1} \phi_m (f^i (p))}, \]
	where $T(p)$ is the period of $p$ and $S(f)$ is the set of all the hyperbolic periodic point of $f$. If $0\leq t<t_0$, then there are no equilibrium states for $(f,t\phi_{m})$ and if $t \geq t_0$, there is an equilibrium state with zero entropy. 
\end{proposition}

\begin{proof}
	Since $P(f,t\phi_{m})$ is continuous with respect to $t$, $\min \lbrace t\ |\  P(f,t\phi_{m})=0 \rbrace$ can be achieved. By Theorem \ref{PressureFormula},
	\[ P (f, t\phi_m) = \sup_{p \in \tmop{Per} (f)} \{\lambda^+(f, p) + \frac{t}{T(p)}  \sum_{i = 0}^{T(p) - 1} \phi_m (f^i (p))\}. \]
	Hence $P (f, t\phi_m)>0$ is equivalent to the following,
	\[ \sup_{p \in \tmop{Per} (f)} \{\lambda^+(f, p) + \frac{t}{T(p)}  \sum_{i = 0}^{T(p) - 1} \phi_m (f^i (p))\} >0 .\]
	This is equivalent to that there exist some periodic point $p$ such that 
	\[ \lambda^+(f, p) + \frac{t}{T(p)}  \sum_{i = 0}^{T(p) - 1} \phi_m (f^i (p)) >0. \]
	Clearly, $p$ is a hyperbolic periodic. Therefore we have the following,
	
	\begin{eqnarray*}
		t&<& \frac{\lambda^+(f,p)}{-\frac{1}{T(p)}  \sum_{i = 0}^{T(p) - 1} \phi_m (f^i (p))}\\
		  &\leq& \sup_{p \in \tmop{Per} (f)\cap S(f)}\frac{\lambda^+(f,p)}{-\frac{1}{T(p)}  \sum_{i = 0}^{T(p) - 1} \phi_m (f^i (p))}. 
	\end{eqnarray*}
	This proves
	\[ t_0= \sup_{p \in \tmop{Per} (f)\cap S(f)}\frac{\lambda^+(f,p)}{-\frac{1}{T(p)}  \sum_{i = 0}^{T(p) - 1} \phi_m (f^i (p))},\]
	By Theorem \ref{No-ES-positive-entropy}, $(f,t\phi_{m})$ has an equilibrium state if and only if the following holds,
	\[ t\sup_{\mu \in \mathbb{P}_{\tmop{erg}}(f)} \int \phi_m \ d\mu = P(f,t\phi_m). \]
	Since 
	\begin{equation*}
		\sup_{\mu \in \mathbb{P}_{\tmop{erg}}(f)} \int \phi_m \ d\mu =0,
	\end{equation*}
	 $(f,t\phi_m)$ has equilibrium states if and only if $P(f,t\phi_{m})=0$. By the definition of $t_0$, there is no equilibrium states if $0 \leq t<t_0$. Since $\phi_{m}$ is negative, $P(f,t\phi_{m})$ is non-increasing with respect to $t$. For the elliptic periodic point $p$ of period $m$, 
	 \begin{equation*}
	 	h_{\mu_p}+ \int t\phi_{m} \ d\mu_p=0.
	 \end{equation*}
	  By the Variational Principle of topological pressure and the fact that $P(f,t\phi_{m})$ is non-increasing with respect to $t$, we have for all $t\geq t_0$, $P(f,t\phi_{m})=0$. Hence $(f,t\phi_{m})$ will have equilibrium state with zero entropy if $t\geq t_0$. 
\end{proof}

For any elliptic periodic point $q$ with period $l$ which is a factor of $m$ (including fixed points) will also give a different equilibrium state $\mu_q$, so in general, the equilibrium states are not unique.

\renewcommand\refname{Reference}
\bibliographystyle{plain}
\bibliography{EquilibriumNoDS}

\medskip

\noindent
\emph{Xueming Hui}\\
{\small Department of Mathematics, Brigham Young University,
	Provo, UT 84602, USA}\\
{\tt hui@mathematics.byu.edu}
\end{document}